\documentclass[12pt]{amsart}
\usepackage{amsmath,amsfonts,amssymb,amsxtra,amscd,enumerate,amsthm}
\usepackage{color}
\usepackage{latexsym}
\usepackage{epsfig}
 
\usepackage{fullpage}

\newcommand\les{\lesssim}
\newcommand\ges{\gtrsim}

\newcommand{\lan}{\langle}
\newcommand{\ran}{\rangle}

\newcommand\R{\mathbb{R}}
\newcommand\C{\mathbb{C}}
\newcommand\Z{\mathbb{Z}}
\newcommand\N{\mathbb{N}}
\renewcommand\S{\mathbb{S}}
\renewcommand\H{\mathbb{H}}

\newcommand{\calO}{\mathcal O}
\newcommand\W{\mathcal{W}}

\newtheorem{theo}{Theorem}
\numberwithin{theo}{section} 
\newtheorem{lema}[theo]{Lemma}
\newtheorem{prop}[theo]{Proposition}

\newtheorem{conjecture}[theo]{Conjecture}

\newtheorem*{rema}{Remark}

\numberwithin{equation}{section}

\newcommand{\dHe}{\dot{H}^1_{e}}

\begin{document}
\title{Equivariant Schr\"odinger Maps in two spatial dimensions: the $\H^2$ target}

\author{I. Bejenaru} \address{ Department of Mathematics, University
  of California, San Diego, 9500 Gilman Drive, La Jolla, CA  92093-0112}
\email{ibejenaru@ucsd.edu}

\author{A. Ionescu} \address{Department of Mathematics, Princeton
  University, Washington Rd., Princeton, NJ 08540}
\email{aionescu@math.princeton.edu}

\author{C. Kenig} \address{ Department of Mathematics, University of
  Chicago, 5734 S. University Ave, Chicago, IL 60637}
\email{cek@math.uchicago.edu}

\author{D.\ Tataru} \address{Department of Mathematics, The University
  of California at Berkeley, Evans Hall, Berkeley, CA 94720, U.S.A.}
\email{tataru@math.berkeley.edu}

\thanks{I.B. was supported in part by NSF grant DMS-1001676. A. I. was
  partially supported by a Packard Fellowship and NSF grant
  DMS-1065710. C.K. was supported in part by NSF grant DMS-0968742. 
  D.T. was supported in part by the Miller Foundation and by NSF 
grant DMS-0801261}

\begin{abstract} We consider equivariant solutions for the
  Schr\"odinger map problem from $\R^{2+1}$ to $\H^2$ with finite energy
   and show that they are global in time and scatter.

\end{abstract}

\maketitle

\section{Introduction}
The Schr\"odinger map equation in $\R^{2+1}$ with values into $\S_\mu
\subset \R^3$ is given by
\begin{equation}
  u_t = u \times_\mu \Delta u, \qquad u(0) = u_0
  \label{SM}\end{equation}
where $\mu=\pm 1$, the connected Riemannian manifolds $S_\mu$,
\begin{equation}\label{defi1}
  \begin{split}
    &S_1=\mathbb{S}^2=\{y=(y_0,y_1,y_2)\in\mathbb{R}^3:y_1^2+y_2^2+y_3^2=1\};\\
    &S_{-1}=\mathbb{H}^2=\{y=(y_0,y_1,y_2)\in\mathbb{R}^3:-y_1^2-y_2^2+y_3^2=1,\,y_3>0\},
  \end{split}
\end{equation}
with the Riemannian structures induced by the Euclidean metric
$\mathbf{g}_1=dy_0^2+dy_1^2+dy_2^2$ on $S_1$, respectively the
Minkowski metric $\mathbf{g}_{-1}=-dy_0^2+dy_1^2+dy_2^2$ on
$S_{-1}$. Thus $S_1$ is the 2-dimensional sphere $\mathbb{S}^2$, while
$S_{-1}$ is the 2-dimensional hyperbolic space $\mathbb{H}^2$. With
$\eta_\mu=\mathrm{diag}(1,1,\mu)$, the cross product $\times_\mu$ is
defined by $v\times_\mu w:=\eta_\mu\cdot(v\times w)$.

This equation admits a conserved energy,
\[
E(u) = \frac12 \int_{\R^2} |\nabla u|^2_\mu dx
\]
and is invariant with respect to the dimensionless scaling
\[
u(t,x) \to u(\lambda^2 t, \lambda x).
\]
The energy is invariant with respect to the above scaling, therefore
the Schr\"odinger map equation in $\R^{2+1}$ is {\em energy critical}.

The local theory for classical data was established in \cite{SuSuBa}
and \cite{Ga}. We recall
\begin{theo}[McGahagan] \label{clasic} If $u_0 \in \dot H^1 \cap \dot
  H^3$ then there exists a time $T>0$, such that \eqref{SM} has a
  unique solution in $L^\infty_t \dot ([0,T]:\dot H^1 \cap \dot H^3)$.
\end{theo}

The local and global in time of the Schr\"odinger map problem with
small data has been intensely studied for the case $\mu=1$
corresponding to $\S^2$ as target, see \cite{Be}, \cite{Be2},
\cite{bik}, \cite{BIKT}, \cite{csu}, \cite{IoKe2}, \cite{IoKe3}.  The
state of the art result for the problem with small data was
established by the authors in \cite{BIKT} where they proved that
classical solutions (and in fact rough solutions too) with small
energy are global in time.  These results are expected to extend to
the case $\mu=-1$, corresponding to $\H^2$ as a target.

To gain some intuition about  the large data problem, one needs to describe the
solitons for \eqref{SM}. The solitons for this problem are the
harmonic maps, which are solutions to $u \times \Delta u=0$.  Since
$\H^2$ is negatively curved there are no finite energy nontrivial
harmonic maps.  In the case of $\S^2$ there are finite energy harmonic
maps, but they cannot have arbitrary energy.  The trivial solitons are
points, i.e. $u=Q$ for some $Q \in \S^2$ and their energy is $0$. The
next energy level admissible for solitons is $4 \pi$; the corresponding soliton
is, up to symmetries, the stereographic projection. Based on this, it is natural to make 
the following 

\begin{conjecture}
a) Global well-posedness and scattering for Schr\"odinger maps
from $\R^2 \times \R$ into $\H^2$  holds for all finite energy data.

b) Global well-posedness and scattering for  Schr\"odinger maps
from $\R^2 \times \R$ into $\S^2$ holds for all data with energy below $4\pi$.
\end{conjecture}

In full generality this remains an open problem. Recently, some
progress was made for the problem with large data in the case of
$\S^2$. Smith established in \cite{Sm} a conditional result for global
existence of smooth Schr\"odinger maps with energy $ < 4\pi$.

In this article we confine ourselves to a class of {\em equivariant}
Schr\"odinger maps.  These are indexed by an integer $m$ called the
equivariance class, and consist of maps of the form
\begin{equation} \label{equiv} u(r,\theta) = e^{m \theta R} \bar{u}(r)
\end{equation} 
Here $R$ is the generator of horizontal rotations, which can be
interpreted as a matrix or, equivalently, as the operator below
\[
R = \left( \begin{array}{ccc} 0 & -1 & 0 \\ 1 & 0 & 0 \\ 0 & 0 &
    0 \end{array} \right) , \qquad R u = \overrightarrow{k} \times_\mu
u.
\]
Here and thereafter we denote by $\overrightarrow{i},
\overrightarrow{j}, \overrightarrow{k}$ the standard orthonormal basis
in $\R^3$, i.e. the vectors with coordinate representation $(1,0,0),
(0,1,0)$ respectively $(0,0,1)$. The case $m=0$ corresponds to radial
symmetry.

The energy for equivariant maps takes the following form:
\begin{equation} \label{energy} E(u) = \pi \int_{0}^\infty \left(
    |\partial_r \bar{u}(r)|^2_\mu +
    \frac{m^2}{r^2}(\bar{u}_1^2(r)+\bar{u}_2^2(r)) \right) r dr
\end{equation}
If $m \ne 0$, then $E(u) < \infty$ implies better information about
the behavior of $u$ versus the radial case $m=0$, in particular it
implies that $u_1$ and $u_2$ have limit zero as $r \rightarrow 0$ and
$r \rightarrow \infty$.

The global regularity question in the case $m=0$ and target $\S^2$,
corresponding to radial symmetry, has been considered recently by
Gustafson and Koo, see \cite{gk}. The global regularity in the case $m
=1$ and target $\S^2$ was considered by the authors in \cite{BIKT2}
where they have shown that the $1$-equivariant solutions of \eqref{SM}
with energy $< 4 \pi$ are globally well-posed.

In this paper we consider the case when the target manifold is $\H^2$
and prove the following

\begin{theo} \label{MT} i) Let $\mu=-1$, $m \ne 0$ and $u_0 \in \dot
  H^1 \cap \dot H^3$ be an $m$-equivariant function. Then \eqref{SM}
  has a unique global in time solution $u \in L^\infty(\R:\dot H^1
  \cap \dot H^3)$.  In addition $\nabla u$, in a particular frame,
  scatters to the free solution of a particular linear Schr\"odinger
  equation.
  
  ii) The above solution is Lipschitz continuous with respect to the
  initial data in $\dot H^1$. In particular if $u_0 \in \dot H^1$ is a
  $m$-equivariant function, $m \ne 0$ then \eqref{SM} has a global
  solution $u(t) \in L^\infty \dot H^1$ defined as the unique limit of
  smooth solutions in $\dot H^1 \cap \dot H^3$. Scattering also holds
for this solution in a suitable frame.
\end{theo}

The statement of the scattering cannot be made precise at this
time. We need to introduce a moving frame on $\H^2$, write the
equation of the coordinates of $\nabla u$ in that frame and identify
there the linear part of the Schr\"odinger equation. This will be
carried out in Section \ref{seccoulomb}.

The result in Theorem \ref{MT} is natural since the failure of the
well-posedness of \eqref{SM} is expected to be closely related to the
existence of finite energy harmonic maps.  In the case of $\H^2$ there
are no harmonic maps, so no obstacles are present.  In the case of
$\S^2$ ($\mu=1$) the lowest energy nontrivial is $4\pi$ and it was
shown in \cite{MRR} that blow-up can occur for maps with energy
$4\pi+$.

\subsection{Definitions and notations.}
\label{defnot}

While at fixed time our maps into the sphere or the hyperbolic space
are functions defined on $\R^2$, the equivariance condition allows us
to reduce our analysis to functions of a single variable $|x|=r \in
[0,\infty)$.  One such instance is exhibited in \eqref{equiv} where to
each equivariant map $u$ we naturally associate its radial component
$\bar u$.  Some other functions will turn out to be radial by
definition, see, for instance, all the gauge elements in Section
\ref{seccoulomb}.  We agree to identify such radial functions with the
corresponding one dimensional functions of $r$.  Some of these
functions are complex valued, and this convention allows us to use the
bar notation with the standard meaning, i.e. the complex conjugate.

Even though we work mainly with functions of a single spatial variable
$r$, they originate in two dimensions. Therefore, it is natural to
make the convention that for the one dimensional functions all the
Lebesgue integrals and spaces are with respect to the $rdr$ measure,
unless otherwise specified.
 
Since equivariant functions are easily reduced to their
one-dimensional companions via \eqref{equiv}, we introduce the one
dimensional equivariant version of $\dot H^1$,
\begin{equation}\label{defhe}
  \| f \|_{\dHe}^2 = \| \partial_r f\|_{L^2(rdr)}^2 + m^2 \| r^{-1} f \|_{L^2(rdr)}^2.
\end{equation}
This is natural since for functions $u: \R^2 \to \R^2$ with
$u(r,\theta) = e^{m \theta R} \bar u(r)$ (here $R u = \overrightarrow{k}
\times u$ or, as a matrix, it is the upper left $2 \times 2$ block of
the original matrix $R$) we have
\[
\| u \|_{\dot H^1}= (2\pi)^\frac12 \| \bar u \|_{\dHe}.
\]
It is important to note that functions in $\dHe$ enjoy the following
properties: they are continuous and have limit $0$ both at $r=0$ and
$r=\infty$, see \cite{gkt1} for a proof.

We introduce $\dot H_e^{-1}$ as the dual space to $\dHe$ with respect
to the $L^2$ pairing, i.e.
\[
\| f \|_{\dot H^{-1}_e} = \sup_{\| \phi \|_{\dHe}=1} \lan f,\phi \ran
\]
The elements from $\dot H^{-1}_e$ can be represented in the form 
$f = \partial_r f_1 + r^{-1} f_2$ with $f_1,f_2 \in L^2$.

Three operators which are often used on radial functions are
$[\partial_r]^{-1}, [r^{-m} \bar \partial_r]^{-1}$ and
$[r \partial_r]^{-1}$ defined as
\[
\begin{split}
  [\partial_r]^{-1} f(r) & = - \int_{r}^\infty f(s) ds, \quad  [r^{-m}\bar \partial_r]^{-1} f(r) = \int_{0}^r f(s) s^m ds \\
  [r \partial_r]^{-1}f(r) & = - \int_{r}^\infty \frac{1}s f(s) ds
\end{split}
\]
A direct argument shows that
\begin{equation} \label{rdrm}
  \begin{split}
    & \| [r\partial_r]^{-1}f \|_{L^p} \lesssim_p \| f \|_{L^p}, \qquad 1 \leq p < \infty, \\
    & \| r^{-m-1} [r^{-m} \bar \partial_r]^{-1}f \|_{L^p} \lesssim_p \| f \|_{L^p}, \qquad 1 < p \leq \infty, \\
    & \| [\partial_r]^{-1}f \|_{L^2} \lesssim \| f \|_{L^1}.
  \end{split}
\end{equation}

The equivariance properties of the functions involved in this paper
requires that the two-dimensional Fourier calculus is replaced by the
Hankel calculus for one-dimensional functions which we recall below.

For $k \geq 0$ integer, let $J_k$ be the Bessel function of the first
kind,
\[
J_k(r)= \frac1\pi \int_0^\pi \cos(n\tau-r\sin\tau) d\tau
\]
If $H_k=\partial_r^2 + \frac1r \partial_r - \frac{k^2}{r^2}$, then
$J_k$ solves $H_k J_k = -J_k$.

We recall some formulas involving Bessel functions
\begin{equation} \label{derB}
  \partial_r J_k = \frac12 (J_{k-1}-J_{k+1}), \quad 
  (r^{-1}\partial_r)^m \left( \frac{J_k}{r^k} \right) = (-1)^m \frac{J_{k+m}}{r^{k+m}}, 
\end{equation}
where $J_{-k}=(-1)^k J_k$.

For each $k \geq 0$ integer one defines the Hankel transform $\mathcal
F_k$ by
\[
\mathcal F_k f (\xi) = \int_0^\infty J_k(r \xi) f(r) rdr
\]
The inversion formula holds true
\[
f(r) = \int_0^\infty J_k(r\xi) \mathcal F_k f(\xi) \xi d\xi
\]
The Plancherel formula holds true, hence in particular, the Hankel
transform is an isometry.

For a radial function $f$ and for an integer $k$ we define its
two-dimensional extension
\begin{equation} \label{Ddef} R_k(r,\theta)=e^{ik\theta} f(r)
\end{equation}
If $f \in L^2$ then $R_k f \in L^2$; if $R_k f$ has additional
regularity, this is easily read in terms of $\mathcal F_k f$.  Indeed
for any $s \geq 0$ integer the following holds true
\begin{equation} \label{Freg} R_k f \in \dot H^s \Leftrightarrow \xi^s
  \mathcal F_k f \in L^2
\end{equation}
For even values of $s$ this is a consequence of $\Delta R_k f= R_k H_k
f$, while for odd values of $s$ it follows by interpolation.

By direct computation, we also have that for $k \ne 0$,
\begin{equation} \label{Frg1} R_k f \in \dot H^1 \Leftrightarrow f \in
  \dHe, \qquad R_0 f \in \dot H^1 \Leftrightarrow \partial_r f \in
  L^2.
\end{equation}

We will use the following result
\begin{lema} \label{LBE}

  i) If $f \in L^2$ is such that $H_k f \in L^2$, with $k \ne 1$, then
  the following holds true
  \[
  \|\partial_r^2 f\|_{L^2} + \| \frac{\partial_r f}{r} \|_{L^2} + k
  \|\frac{f}{r^2} \|_{L^2} \les \| H_k f \|_{L^2}
  \]
  
  ii) If $f \in L^2$ is such that $H_1 f \in L^2$, then the following
  holds true
  \[
  \|\partial_r^2 f\|_{L^2} + \| \frac{\partial_r f}{r} - \frac{f}{r^2}
  \|_{L^2} \les \| H_1 f \|_{L^2}
  \]
  iii) If $f \in L^2$ is such that $\partial_r H_0 f \in L^2$, then
  the following holds true
  \[
  \|\partial_r^3 f\|_{L^2} + \| \frac{\partial_r^2 f}{r} -
  \frac{\partial_r f}{r^2} \|_{L^2} \les \| \partial_r H_0 f \|_{L^2}
  \]
  iv) If $f \in L^2$ is such that $H_1 f \in \dHe$, then the following
  holds true
  \[
  \|\partial_r^3 f\|_{L^2} + \| \frac{\partial_r^2 f}{r} \|_{L^2} + \|
  \frac{\partial_r f}{r^2} - \frac{f}{r^3} \|_{L^2} \les \| H_1 f
  \|_{\dHe}
  \]
  v) If $f \in L^2$ is such that $H_2 f \in \dHe$, then the following
  holds true
  \[
  \|\partial_r^3 f\|_{L^2} +\|\frac{\partial_r^2 f}{r} -
  \frac{\partial_r f}{r^2} \|_{L^2} + \|\frac{ \partial_r f}{r^2} -
  \frac{2 f}{r^3} \|_{L^2} \les \|  H_2 f \|_{\dHe}
  \]
  vi) If $f \in L^2$ is such that $H_k f \in \dHe$, with $k \geq 3$,
  then the following holds true
  \[
  \|\partial_r^3 f\|_{L^2} + \|\frac{\partial_r^2 f}{r}\|_{L^2} +\|
  \frac{\partial_r f}{r^2}\|_{L^2} + \| \frac{f}{r^3} \|_{L^2} \les \|
  H_k f \|_{\dHe}
  \]
  vii) If $f, \partial_r f \in L^2$, then for any $2 \leq p < +\infty$
  the following holds true
  \[
  \| f \|_{L^p} \les_p \| \partial_r f \|_{L^2} + \| f \|_{L^2}
  \]
  viii) If $f, H_k f \in L^2$, with $k \geq 0$, then for any $2 \leq p
  < +\infty$ the following holds true
  \[
  \| \partial_r f \|_{L^p} \les_p \| H_k f \|_{L^2} + \| f \|_{L^2}
  \]
 
\end{lema}

\begin{proof} Part i) for $k \in \{0,2\}$ are established in Lemma 1.3
  in \cite{BIKT2}, and the general result for all $k \geq 3$ follows
  along the same lines.

  For part ii) we use the inversion formula for $f$ and \eqref{derB}
  to compute
  \[
  \partial_r^2 f = \int (J_3 - 3 J_1)(r\xi) \xi^2 \mathcal{F}_1 f(\xi)
  \xi d\xi
  \]
  and the first part of the estimate follows. The estimate for the
  second term follows from the form of $H_1 f$.

  For part iii) we proceed as above, i.e. use the inversion formula
  for $f$ and \eqref{derB} to write
  \[
  \partial_r^3 f = \int (J_3 - 3 J_1)(r\xi) \xi^3 \mathcal{F}_0 f(\xi)
  \xi d\xi
  \]
  and conclude with the estimate for $\| \partial_r^3 f \|_{L^2}$,
  while the estimate for the second term follows from the expression
  of $\partial_r H_0$.

  Parts iv)-vi) follow in a similar manner by using the Hankel
  transform and \eqref{derB} to derive the estimates.  The details are
  left to the reader.

  vii) and viii) are consequences of the standard Sobolev embeddings.

\end{proof}

\subsection{A few calculus rules}

We recall that given $\mu= \pm 1$ and two vectors
$v={}^t(v_1,v_2,v_3)$ and $w={}^t(w_1,w_2,w_3)$ in $\mathbb{R}^3$,
their inner product is defined as
\begin{equation}\label{defi2}
  v\cdot_\mu w=\mathbf{g}_{-1}(v,w)={}^tv\cdot\eta_\mu\cdot w=v_1w_1+v_2w_2 +\pm v_3 w_3,
\end{equation}
where $\eta_\mu=\mathrm{diag}(1,1,\mu)$. We define also the cross
product
\begin{equation}\label{defi3}
  v\times_\mu w:=\eta_\mu\cdot(v\times w), 
\end{equation}
where $v\times w$ denotes the usual vector product of vectors in
$\mathbb{R}^3$. Simple computations show that, for $\mu=\pm 1$ and
$v,w\in\mathbb{R}^3$
\begin{equation}\label{products}
  \begin{split}
    &v\cdot_\mu(v\times_\mu w)=w\cdot_\mu(v\times_\mu w)=0,\\
    &(v\times_\mu w)\cdot_\mu(v\times_\mu w)=\mu(v\cdot_\mu v)(w\cdot_\mu w)-\mu(v\cdot_\mu w)^2\\
    & (a \times_\mu b) \cdot_\mu c = a \cdot_\mu (b \times_\mu c)
  \end{split}
\end{equation}

\subsection{Energy estimates} \label{enest} In this section we derive
properties of $u$ from the finiteness of its energy $E(u)$ in
\eqref{energy} in the case $\mu=-1$ (in the case $\mu=1$ the
corresponding estimates are trivial as all terms come with $+$ sign).
We recall that
\[
E(u) = \pi \int_{0}^\infty \left( |\partial_r \bar{u}_1(r)|^2 +
  |\partial_r \bar{u}_2(r)|^2 - |\partial_r \bar{u}_3(r)|^2 +
  \frac{m^2}{r^2}(\bar{u}_1^2(r)+\bar{u}_2^2(r)) \right) r dr
\]
Since $u_1 \partial_r \bar u_1 + u_2 \partial_r \bar u_2
=u_3 \partial_r \bar u_3 $ and $\bar u_3^2=1+ \bar u_1^2 + \bar u_2^2$
it follows that
\[
E(u) = \pi \int_{0}^\infty \left( \frac{|\partial_r \bar{u}_1(r)|^2 +
    |\partial_r \bar{u}_2(r)|^2}{\bar u_3^2} + (\frac{\bar
    u_2 \partial_r \bar u_1 - \bar u_2 \partial_r \bar u_1}{\bar
    u_3})^2 + \frac{m^2}{r^2}(\bar{u}_1^2(r)+\bar{u}_2^2(r)) \right) r
dr
\]
We also have that
\[
\begin{split}
  \bar u_1^2(r) + \bar u_2^2(r) - \bar u_1^2(1) - \bar u_2^2(1) & = \int_1^r \partial_r (\bar u_1^2 + \bar u_2^2) ds \\
  & \les \int_1^r (|\bar u_1|+|\bar u_2|)(|\partial_r \bar u_1|+|\partial_r \bar u_2|) ds \\
  & \les \int_1^r (|\bar u_1|+|\bar u_2|) \frac{|\bar u_1|+|\bar u_2|}s \frac{|\partial_r \bar u_1|+|\partial_r \bar u_2|}{\bar u_3} s ds \\
  & \les \sup_{s \in [1,r]} (|\bar u_1(s)|+|\bar u_2(s)|) E(u)
\end{split}
\]
from which we conclude that $\sup_{r \in (0, \infty)}
|u_1(r)|+|u_2(r)| \les \bar u_1(1) + \bar u_2(1) + E(u)$.  Therefore
$\sup_{r \in (0, \infty)} | \bar u_3(r)| \les m + \bar u_1(1) + \bar
u_2(1) + E(u)$, hence from the last expression of $E(u)$ we obtain
that $\bar u_1, \bar u_2 \in \dHe$. In particular it follows that
$\bar u_1(0)=\bar u_2(0)=0$ (in the sense that the limits exists and
equal $0$), hence rewriting the above argument on $(0,r]$ instead
gives $|\bar u_1|_{L^\infty} + |\bar u_2|_{L^\infty} \les E(u)$,
$|\bar u_3|_{L^\infty} \les m+E(u)$.  Recalling the last expression of
$E(u)$ we obtain
\[
\| \bar u_1 \|_{\dHe} + \| \bar u_2 \|_{\dHe} \les E(u)^\frac12
(m+E(u))
\]
In addition we obtain $\bar u_3 -1 \in \dHe$ with
\[
\| \bar u_3 -1 \|_{\dHe} \les E(u)^\frac12 (m+E(u))
\]

\section{The Coulomb gauge representation of the equation}
\label{seccoulomb}

In this section we rewrite the Schr\"odinger map equation for
equivariant solutions in a gauge form.  This approach originates in
the work of Chang, Shatah, Uhlenbeck~\cite{csu}. However, our analysis
is closer to the one in \cite{bik} and \cite{BT-SSM}. The computations
in subsections \ref{CG} and \ref{COG} follow exactly the same lines as the one used in \cite{BIKT2}.
Then we fix $\mu=-1$ as the analysis becomes more specific to this case. 

\subsection{The Coulomb gauge} \label{CG}

The computations below are at the formal level as we are not yet
concerned with the regularity of the terms involved in writing various
identities and equations. Implicitly we use only the information $u
\in \dot H^1$. In subsection \ref{sreg} we prove that if $u \in \dot
H^3$ then all the gauge elements, their compatibility relations and
the equations they obey are meaningful in the sense that they involve
terms which are at least at the level of $L^2$.

We let the differentiation operators
$\partial_0,\partial_1,\partial_2$ stand for
$\partial_t, \partial_r, \partial_\theta$ respectively. Our strategy
will be to replace the equation for the Schr\"odinger map $u$ with
equations for its derivatives $\partial_1 u$, $\partial_2 u$ expressed
in an orthonormal frame $v,w \in T_u \S_\mu$. We choose $v \in T_u \S_\mu$
such that $v \cdot_\mu v=1$ and define $w = u \times_\mu v \in T_u \S_\mu$; to
summarize
\begin{equation}
  v \cdot_\mu v=1, \qquad v \cdot_\mu u =0, \qquad w=u \times_\mu v
\end{equation}
From this, we obtain
\begin{equation}
  w \cdot_\mu v=0, \qquad w \cdot_\mu w=1, \qquad v \times_\mu w= \mu u, \qquad w \times_\mu u=v
\end{equation}
Since $u$ is $m$-equivariant it is natural to work with
$m$-equivariant frames, i.e.
\[
v = e^{m \theta R} \bar{v}(r), \qquad w = e^{m \theta R} \bar{w}(r).
\]
where $\bar v, \bar w$ (as well as $\bar u$ from \eqref{equiv}) are
unit vectors in $\R^3$.

Given such a frame we introduce the differentiated fields $\psi_k$ and
the connection coefficients $A_k$ by
\begin{equation}\label{connection}
  \begin{split}
    \psi_k = \partial_k u \cdot_\mu v + i \partial_k u \cdot_\mu w, \qquad A_k
    = \partial_k v \cdot w.
  \end{split}
\end{equation}
Due to the equivariance of $(u,v,w)$ it follows that both $\psi_k$ and
$A_k$ are spherically symmetric (therefore subject to the conventions
made in Section \ref{defnot}). Conversely, given $\psi_k$ and $A_k$ we
can return to the frame $(u,v,w)$ via the ODE system:
\begin{equation}
  \label{return}
  \left\{ \begin{array}{l}
      \partial_k u = (\Re {\psi_k}) v  + (\Im{\psi_k}) w
      \cr
      \partial_k v = - \mu (\Re{\psi_k}) u + A_k w
      \cr
      \partial_k w = - \mu (\Im{\psi_k}) u - A_k v
    \end{array} \right.
\end{equation}

If we introduce the covariant differentiation
\[
D_k = \partial_k + i A_k, \ \ k \in \{0,1,2\}
\]
it is a straightforward computation to check the compatibility
conditions:
\begin{equation} \label{compat} D_l \psi_k = D_k \psi_l, \ \ \
  l,k=0,1,2.
\end{equation}
The curvature of this connection is given by
\begin{equation}\label{curb}
  D_l D_k - D_k D_l = i(\partial_l A_k - \partial_k
  A_l) = i \mu \Im{(\psi_l \bar{\psi}_k)}, \ \ \  l,k=0,1,2.
\end{equation}
An important geometric feature is that $\psi_2, A_2$ are closely
related to the original map. Precisely, for $A_2$ we have:
\begin{equation}
  A_2 = m ( \overrightarrow{k} \times_\mu v) \cdot_\mu w=  
 m \overrightarrow{k} \cdot_\mu (v \times_\mu w)= m \overrightarrow{k} \cdot_\mu  (\mu u) = m u_3
  \label{a2u3}\end{equation}
and, in a similar manner,
\begin{equation}
  \psi_2 = m \mu (w_3-iv_3)
  \label{psi2vw3}\end{equation}
Since the $(u,v,w)$ frame is orthonormal, it follows that $|\psi_2|^2
= m^2(u_1^2 + u_2^2)$ and the following important conservation law
\begin{equation} \label{cons} |\psi_2|^2 + \mu A_2^2 = \mu m^2
\end{equation}

Now we turn our attention to the choice of the $(\bar v,\bar w)$ frame
at $\theta = 0$. Here we have the freedom of an arbitrary rotation
depending on $t$ and $r$. In this article we will use the Coulomb
gauge, which for general maps $u$ has the form $\text{div } A = 0$.
In polar coordinates this is written as $\partial_1 A_1 + r^{-2}
\partial_2 A_2 = 0$. However, in the equivariant case $A_2$ is radial,
so we are left with a simpler formulation $A_1 = 0$, or equivalently
\begin{equation}
  \partial_r \bar v \cdot_\mu \bar w=0
  \label{coulomb}\end{equation}
which can be rearranged into a convenient ODE as follows
\begin{equation} \label{cgeq}
  \partial_r  \bar v = \mu (\bar v \cdot_\mu \bar u) \partial_r \bar u
  - \mu (\bar v \cdot_\mu \partial_r \bar u) \bar u
\end{equation}
The first term on the right vanishes and could be omitted, but it is
convenient to add it so that the above linear ODE is solved not only
by $\bar v$ and $\bar w$, but also by $\bar u$. Then we can write an
equation for the matrix $ \calO = (\bar v, \bar w,\bar u)$:
\begin{equation} \label{cgeq-m}
  \partial_r \calO = M \eta_\mu \calO, \qquad M = \partial_r \bar u \wedge \bar u : = 
  \partial_r \bar u \otimes \bar u - \bar u \otimes \partial_r \bar u
\end{equation}
with an antisymmetric matrix $M$.

An advantage of using the Coulomb gauge is that it makes the
derivative terms in the nonlinearity disappear. Unfortunately, this
only happens in the equivariant case, which is why in \cite{BIKT} we
had to use a different gauge, namely the caloric gauge.

The ODE \eqref{cgeq} needs to be initialized at some point.  A change
in the initialization leads to a multiplication of all of the $\psi_k$
by a unit sized complex number. This is irrelevant at fixed time, but
as the time varies we need to be careful and choose this
initialization uniformly with respect to $t$, in order to avoid
introducing a constant time dependent potential into the equations via
$A_0$.  Since in our results we start with data which converges
asymptotically to $\vec{k}$ as $r \to \infty$, and the solutions
continue to have this property, it is natural to fix the choice of
$\bar{v}$ and $\bar{w}$ at infinity,
\begin{equation}
  \label{bcvw}
  \lim_{r \to \infty} \bar v(r,t) = \vec{i} , \qquad \lim_{r \to \infty} 
  \bar w(r,t) = -\mu \vec{j}
\end{equation}

The existence of a unique solution $\bar v \in C((0,+\infty):\R^3)$ of
\eqref{cgeq} satisfying \eqref{bcvw} is standard, we skip the details.
Moreover the solution is continuous with respect to $u$
in the following sense
\begin{equation} \label{gcont}
\| \bar v- \bar{\tilde v} \|_{L^\infty} \les \| u - \tilde u \|_{\dot H^1}
\end{equation}

\subsection{ Schr\"odinger maps in the Coulomb gauge} \label{COG}

We are now prepared to write the evolution equations for the
differentiated fields $\psi_1$ and $\psi_2$ in \eqref{connection}
computed with respect to the Coulomb gauge.

Writing the Laplacian in polar coordinates, a direct computation using
the formulas \eqref{connection} shows that we can rewrite the
Schr\"odinger Map equation \eqref{SM} in the form
\begin{equation}\label{psizero}
  \psi_0 = i \left(D_1 \psi_1 + \frac{1}{r} \psi_1 + \frac{1}{r^2} D_2
    \psi_2\right)
\end{equation}
Applying the operators $D_1$ and $D_2$ to both sides of this equation
and using the relation \eqref{curb} for $l,k = 1,2$ we obtain
\begin{equation} \label{psiab}
  \begin{split}
    D_1 \psi_0 = & \ i\left( D_1(D_1+\frac1r) \psi_1 + \frac{1}{r^2}
      D_2 D_1 \psi_2\right) - \frac{\mu}{r^2} \Im{(\psi_1 \bar{\psi}_2)}
    \psi_2
    \\
    D_2 \psi_0 = & \ i\left( (D_1+\frac1r)D_2 \psi_1 + \frac{1}{r^2}
      D_2 D_2 \psi_2\right) - \mu \Im{(\psi_2 \bar{\psi}_1)} \psi_1
  \end{split}
\end{equation}
Using now \eqref{compat} for $(k,l)=(0,1)$ respectively $(k,l)=(0,2)$
on the left and for $(k,l)=(1,2)$ on the right we can derive the
evolution equations for $\psi_m$, $ m=1,2$:
\begin{equation} \label{dtpsiab}
  \begin{split}
    D_0 \psi_1 = & \ i\left( D_1(D_1+\frac1r) + \frac{1}{r^2} D_2
      D_2\right) \psi_1 - \frac{\mu}{r^2} \Im{(\psi_1 \bar{\psi}_2)}
    \psi_2
    \\
    D_0 \psi_2 = & \ i\left( (D_1+\frac1r)D_1 + \frac{1}{r^2} D_2
      D_2\right) \psi_2 - \mu \Im{(\psi_2 \bar{\psi}_1)} \psi_1
  \end{split}
\end{equation}
In our set-up all functions are radial and we are using the the
Coulomb gauge $A_1 = 0$. Then these equations take the simpler form
\begin{equation*} \label{smg}
  \begin{split}
    \partial_t \psi_1 + i A_0 \psi_1 = & i \Delta \psi_1 - i
    \frac{1}{r^2} A_2^2 \psi_1 -i \frac1{r^2} \psi_1
    +  \frac2{r^3} A_2 \psi_2 - \frac{\mu}{r^2} \Im{(\psi_1 \bar{\psi}_2)} \psi_2 \\
    \partial_t \psi_2 + i A_0 \psi_2 = & i \Delta \psi_2 - i
    \frac{1}{r^2} A_2^2 \psi_2 - \mu \Im{(\psi_2 \bar{\psi}_1)} \psi_1
  \end{split}
\end{equation*}
The two variables $\psi_1$ and $\psi_2$ are not independent.  Indeed,
the relations \eqref{compat} and \eqref{curb} for $(k,l)=(1,2)$ give
\begin{equation} \label{comp}
  \partial_r A_2= \mu \Im{(\psi_1 \bar{\psi}_2)}, \qquad \partial_r \psi_2 = i A_2
  \psi_1
\end{equation}
which at the same time describe the relation between $\psi_1$ and
$\psi_2$ and determine $A_2$.

From the compatibility relations involving $A_0$, we obtain
\begin{equation} \label{a0}
  \partial_r A_0= - \frac{\mu}{2r^2} \partial_r (r^2 |\psi_1|^2 - |\psi_2|^2)
\end{equation}
from which we derive
\begin{equation}
  A_0 = - \frac{\mu}2 \left( |\psi_1|^2 - \frac{1}{r^2}|\psi_2|^2\right)
  - \mu [r\partial_r]^{-1} \left( |\psi_1|^2 - \frac{1}{r^2}|\psi_2|^2\right) 
  \label{a0bis}\end{equation}
This is where the initialization of the Coulomb gauge at infinity is
important. It guarantees that $A_0 \in L^p$, provided that $|\psi_1|^2
- r^{-2}|\psi_2|^2 \in L^p$ for $1 \leq p < \infty$. In particular,
without any additional regularity assumptions, we know that $A_0 \in
L^1$.  A direct computation using integration by parts gives that
\begin{equation} \label{A0c} \int A_0(r) rdr =0.
\end{equation}

The system satisfied by $\psi_1$ and $\frac{\psi_2}r$ (this being in
fact the correct variable instead of $\psi_2$) is given by:
\begin{equation*}
  \begin{split}
    (i \partial_t + \Delta - \frac{m^2+1}{r^2}) \psi_1 + \frac{2 \mu mi}{r^2}
    \frac{\psi_2}r= & A_0 \psi_1 + \frac{A_2^2-m^2}{r^2} \psi_1
    + 2i \frac{A_2+\mu m}{r^3} \psi_2 - i \mu \Im{(\psi_1 \frac{\bar{\psi}_2}r}) \frac{\psi_2}r \\
    (i \partial_t + \Delta - \frac{m^2+1}{r^2}) \frac{\psi_2}{r} -
    \frac{2\mu mi}{r^2} \psi_1= & A_0 \frac{\psi_2}r + \frac{A_2^2-m^2}{r^2}
    \frac{\psi_2}r - 2i \frac{A_2+\mu m}{r^2} \psi_1- i \mu
    \Im{(\frac{\psi_2}r \bar{\psi}_1)} \psi_1
  \end{split}
\end{equation*}
The problem with this system is that its linear part is not decoupled.
This can be remedied by a change of variables. Indeed consider
\[
\psi^-=\psi_1 -i \frac{\psi_2}{r}, \qquad \psi^+=\psi_1 + i
\frac{\psi_2}{r}
\]
It turns out that $\psi^\pm$ satisfy a similar system (described
below) whose linear part is decoupled.  The relevance of the variables
$\psi^{\pm}$ comes also from the following reinterpretation.  If
$\W^{\pm}$ is defined as the vector
\[
\W^{\pm} = \partial_r u \pm \frac{1}{r} u \times \partial_\theta u \in T_u(\S_\mu)
\]
then $\psi^\pm$ is the representation of $W^\pm$ with respect to the
frame $(v,w)$. On the other hand, a direct computation leads to
\[
\begin{split}
  E(u) & = \pi \int_0^\infty \left( |\partial_r \bar{u}|^2 +
    \frac{m^2}{r^2} |\bar{u} \times
    R \bar{u}|^2 \right) rdr \\
  & = \pi \| \bar \W^\pm \|_{L^2}^2 \mp 2 \pi m (\bar u_3(\infty)- \bar
  u_3(0))
\end{split}
\]
where we recall that $u(r,\theta)= e^{m \theta R} \bar{u}(r)$ and
$\bar u_3(\infty)=\lim_{r \rightarrow \infty} \bar u_3(r), \bar
u_3(0)=\lim_{r \rightarrow 0} \bar u_3(r)$ are well-defined since
$\bar u_1, \bar u_2 \in \dHe$ and if $f \in \dHe$ then $\lim_{r
  \rightarrow 0} f(r)=\lim_{r \rightarrow 0} f(r)=0$, see \cite{gkt1}
or \cite{BT-SSM}. From Section \ref{enest} it follows that, in the case $\mu=-1$, 
$\bar u_3(\infty)=\bar u_3(0)=1$. In the case $\mu=1$ one needs the energy restriction
$E(u) < 4\pi$ to obtain that $\bar u_3(\infty)=\bar u_3(0)=1$, see \cite{BIKT2}. 
In both cases we obtain the following identity
\begin{equation} \label{basicpsi} \|\psi^\pm\|_{L^2}^2 = \| \bar
  \W^\pm \|_{L^2}^2 = \frac{E(u)}{\pi}.
\end{equation}

From \eqref{gcont} it follows that the following continuity property holds true
\begin{equation} \label{dpsi}
\| \psi^\pm - \tilde{\psi}^\pm \|_{L^2} \les \| u - \tilde u  \|_{\dot H^1}
\end{equation}

A direct computation yields the following system for $\psi^{\pm}$:
\[
\begin{split}
  ( i \partial_t + H^-_m) \psi^- & = \left( A_0 - 2 \frac{A_2+\mu m}{r^2}
    + \frac{A_2^2-m^2}{r^2} - \frac\mu{r} \Im{(\psi_2 \bar{\psi}_1)} \right) \psi^- \\
  (i \partial_t + H^+_m) \psi^+ & = \left( A_0 + 2 \frac{A_2+\mu m}{r^2} +
    \frac{A_2^2-m^2}{r^2} + \frac\mu{r} \Im{(\psi_2 \bar{\psi}_1)} \right)
  \psi^+
\end{split}
\]
where
\[
H^-_m = \Delta - \frac{(m+\mu)^2}{r^2}, \qquad H^+_m = \Delta - \frac{(m-\mu)^2}{r^2}.
\]
Here and whenever $\Delta$ acts on radial functions, it is known that
$\Delta = \partial_r^2 + \frac1r \partial_r$.  By replacing $\psi_1 =
\psi^\pm \mp i r^{-1} \psi_2$ and using $\mu A_2^2 + |\psi_2|^2=\mu m^2$, we obtain
the key evolution system we work with in this paper,
\begin{equation} \label{psieq} \left\{ \begin{array}{l} (i \partial_t
      + H^{-}_m) \psi^- = (A_0 - 2 \frac{A_2+\mu m}{r^2} -
      \frac{\mu}{r}\Im{(\psi_2 \bar{\psi}^-)}) \psi^- \cr (i \partial_t +
      H^+_m) \psi^+ = ( A_0 + 2 \frac{A_2+\mu m}{r^2} + \frac{\mu}{r}\Im{(\psi_2
        \bar{\psi}^+)}) \psi^+
    \end{array} \right.
\end{equation}
We will use this system in order to obtain estimates for
$\psi^\pm$. The old variables $\psi_1$ and $\frac{\psi_2}r$ are
recovered from
\begin{equation} \label{rela} \psi_1 = \frac{\psi^+ + \psi^-}2, \qquad
  \frac{\psi_2}r = \frac{\psi^+ - \psi^-}{2i}
\end{equation}
From the compatibility conditions \eqref{comp} we derive the formula
for $A_2$
\begin{equation} \label{A2} A_2(r) + \mu m= - \mu \int_0^r
  \frac{|\psi^+|^2-|\psi^-|^2}{4} s ds
\end{equation}

From \eqref{a0bis} $A_0$ is given by
\begin{equation} \label{A0} A_0 = - \frac{\mu}2 \Re(\overline{\psi}^+
  \psi^-) + \mu [r\partial_r]^{-1} \Re(\overline{\psi}^+ \psi^-)
\end{equation}

The compatibility condition \eqref{comp} reduces then to
\begin{equation} \label{compnew}
  \partial_r [ r (\psi^+ - \psi^-)] = - A_2 (\psi^+ + \psi^-)
\end{equation}

Next assume that $\psi^\pm \in L^2$ are given such that they satisfy the
compatibility conditions \eqref{compnew}.  We reconstruct $A_2,
\psi_2, \psi_1$ using the \eqref{rela} and \eqref{A2}.  From
\eqref{A2} and \eqref{compnew} it follows that \eqref{comp} hold true.
From \eqref{A2} it follows that $A_2 \in L^\infty$ and it is
continuous and has limits both at $0$ and $\infty$. From the
definition of $\psi_2$ we have $\frac{\psi_2}r \in L^2$ and from
\eqref{compnew} we derive $\partial_r \psi_2 \in L^2$, hence $\psi_2
\in \dHe$. From this and \eqref{A2} it follows that $\partial_r A_2
\in L^2$, while by invoking \eqref{rdrm} we obtain $\frac{A_2+\mu m}r \in
L^2$, therefore $A_2 + \mu m \in \dHe$.  In particular $A_2(\infty)=\lim_{r
  \rightarrow \infty} A_2(r)=-\mu m$ which implies that $\| \psi^+
\|_{L^2}=\| \psi^- \|_{L^2}$.

In fact one can keep track of a single variable, $\psi^-$ or $\psi^+$
since it contains all the information about the map, provided that the
choice of gauge \eqref{bcvw} was made. To be more precise,
\eqref{comp} gives the following
\begin{equation}
  \label{comp1}
  \partial_r A_2= \mu \Im{(\psi^- \bar{\psi}_2)} + \frac{\mu}r |\psi_2|^2, \qquad \partial_r
  \psi_2 = i A_2 \psi^- - \frac{1}r A_2 \psi_2
\end{equation}
We will show that given $ \psi^- \in L^2$, this system has a
unique solution $A_2+\mu m,\psi_2 \in \dHe$.  From this we can reconstruct
$\psi_1, \psi^+,A_0$. Finally, given $\psi^-$, $A_2$ and $\psi_2$, we
can return to the Schr\"odinger map $u$ via the system \eqref{return}
with the boundary condition at infinity given by
\eqref{bcvw}. Eventually we show that if $\psi^-$ satisfies its
corresponding equation from \eqref{psieq}, then the $u$ obtained is a
Schr\"odinger map. A similar procedure can completely reconstruct $u$
from $\psi^+$.

The reason to keep both variables $\psi^\pm$ (instead of just one) has
to do with the nonlinear analysis of the system \eqref{psieq}. The
reason we want to understand how to recover all information from only
one variable, say $\psi^-$, has to do with the elliptic part of the
profile decomposition in Proposition \ref{pcc}.

\subsection{Fix $\mu=-1$} The theory with $\mu =1$ was developed in
\cite{BIKT2}.  From this point on we fix $\mu=-1$ as the theory
becomes more specific to this case.  When comparing the results
obtained here and those in \cite{BIKT2} the reader may notice a few
differences. First, one sees that $\psi^\pm$ come with operators $H_{m
  \mp \mu}$ and all the consequences associated, see for instance the
regularity below. This is a consequence of the way we chose the limits
$\lim_{r \rightarrow 0} \bar u_3 = \lim_{r \rightarrow \infty} \bar
u_3= -\mu$.  Second, the analytic theory of the system \eqref{comp1}
with $\mu=-1$, see Proposition \ref{a2p2c}, is somehow different then
its counterpart for $\mu=1$. The Cauchy theory in Section \ref{CTH}
and the Concentration compactness argument in Section \ref{concomp}
are very similar. Finally, the arguments in Section \ref{MOM} are
again specific to the case $\mu =-1$, as in particular no restriction
on the size of the energy/mass is needed to rule out the possibility
of blow-up.

\subsection{Regularity of the gauge elements} \label{sreg} In this
section we clarify the regularity of the gauge elements. Our main claim is
the following
\begin{prop} \label{greg} If $u \in \dot H^3$ then $R_{m \pm 1} \psi^\pm \in
  H^2$ and
  \begin{equation} \label{rtr} \| u \|_{\dot H^1 \cap \dot H^3}
    \approx \| R_{m+1} \psi^+ \|_{H^2} + \| R_{m-1} \psi^- \|_{H^2}
  \end{equation}
\end{prop}
The proof of this result will be provided in the Appendix.

Therefore, in the context of $u \in \dot H^1 \cap \dot H^3$, we have
that $R_{m \pm 1} \psi^\pm \in H^2 \subset L^\infty$. The $H^2$ regularity
cannot be extended to (two-dimensional extensions of) $\psi_1$ and
$\frac{\psi_2}r$ since the $\psi^+$ and $\psi^-$ require different
phases for regularity. However, all the Sobolev embeddings are
inherited by $\psi_1$ and $\frac{\psi_2}r$, in particular $\psi_1,
\frac{\psi_2}r \subset L^\infty$.  Since $A_2=u_3$ it follows that
$A_2 \in \dot H^1 \cap \dot H^3$ and $\partial_t A_2 \in H^1$.
Finally by differentiating with respect to $t$ the system
\eqref{cgeq}, one can show that $\partial_t \bar v \in H^1$, hence
$A_0 \in H^1$ which in turn gives $\partial_r A_0 \in L^2$.  With
these in mind, all the compatibility conditions in the previous two
subsections are at least at the level of $L^2$.

\subsection{Recovering the map from $\psi^-$.} In this section we
address the issue of re-constructing the Schr\"odinger map $u$
together with its gauge elements from only one of its reduced
variable, say $\psi^-$.  Reconstructing $\psi_2,A_2$ such that
$\psi_2,A_2-m \in \dHe$ is a unique process; however, the
reconstruction of the actual map with its frame, i.e. of $(u,v,w)$ is
unique provided one prescribes conditions at $\infty$.  The map $u$
satisfies $u(\infty)=\vec{k}$, while the gauge is subjected to the
choice \eqref{bcvw}.

The main result of this section is the following

\begin{prop} \label{recprop} Given $\psi^- \in L^2$, there is a unique map $u: \R^2
  \rightarrow \S^2$ with the property that $\psi^-$ is the
  representation of $\mathcal{W}^-$ relative to a Coulomb gauge
  satisfying \eqref{bcvw}.  This also satisfies $E(u)=\pi \| \psi^-
  \|_{L^2}^2$.
  
If $\tilde \psi^- \in L^2$ and $\tilde u$ is the corresponding map as above, 
then the following holds true
\begin{equation} \label{udif}
E(u-\tilde u) \les \| \psi^- - \tilde \psi^- \|_{L^2}^2
\end{equation}

\end{prop}

Here $\psi^+$ can be reconstructed from $\psi^-$.  Moreover the
equations \eqref{comp1} which we use for reconstruction force the
compatibility condition \eqref{compnew} between $\psi^\pm$.  The
result remains true if we start from $\psi^+$ just that we would start
the reconstruction (described below) from the analogue of the
\eqref{comp1} written in terms of  $\psi^+$. The two problems are in effect
equivalent via an inversion. The uniqueness of the reconstruction
guarantees that starting from either $\psi^+$ or $\psi^-$ (which are
assumed to be compatible) gives the same $u$.
 
The proof consists of several steps. The first one deals with
recovering the two gauge elements $\psi_2,A_2$ from $\psi^-$ by using
the system \eqref{comp1}.

\begin{lema} \label{a2p2c} Given $\psi^- \in L^2$, the system \eqref{comp1} has a unique
  solution $(A_2, \psi_2)$ satisfying $\psi_2, A_2 -m \in \dHe$. This
  solution satisfies
  \begin{equation} \label{2est} \| \psi_2 \|_{\dHe} + \| A_2-m
    \|_{\dHe} + \| \frac{A_2-m}r \|_{L^1(dr)} \les \| \psi^- \|_{L^2}(m+\| \psi^- \|_{L^2}^2)
  \end{equation}
  In addition we have the following properties:
  
  i) given $\epsilon > 0$, and $R$  such that $\| \psi^-
  \|_{L^2(\R \setminus [R^{-1},R])} \leq \epsilon$, then the following holds
  true
  \begin{equation} \label{loc2est} \| \psi_2 \|_{\dHe(\R \setminus
      [R^{-1},\epsilon^{-1} R])} + \| A_2 -m \|_{\dHe(\R \setminus
      [R^{-1},\epsilon^{-1} R])} \les \epsilon \| \psi^- \|_{L^2}
  \end{equation}
  ii) if $(\tilde A_2, \tilde \psi_2)$ is another solution (as above)
  to \eqref{comp1} with $\tilde \psi^-$, then
  \begin{equation} \label{1estdif} \| \psi_2 - \tilde \psi_2 \|_{\dHe}
    + \| A_2 - \tilde A_2 \|_{\dHe} \les \| \psi^- - \tilde \psi^-
    \|_{L^2}
  \end{equation}
  iii) if $(\tilde A_2, \tilde \psi_2)$ satisfy $\tilde \psi_2, \tilde
  A_2-m \in \dHe $ and solve
  \begin{equation} \label{sys2ap}
    \begin{split}
      \partial_r \tilde \psi_2 = & \ i \tilde A_2 \tilde \psi^- -
      \frac{1}r \tilde A_2 \tilde \psi_2 + E_1 \\
      \partial_r \tilde A_2= & - \ \Im{(\tilde \psi^- \bar{\tilde
          \psi}_2)}-\frac{1}r (\tilde A_2^2-m^2) + E_2
    \end{split}
  \end{equation}
  where $\| |E_1|+ |E_2|\|_{L^1(dr)+L^2} \les \epsilon$ then
  \begin{equation} \label{2estdif} \| \psi_2 - \tilde \psi_2 \|_{\dHe}
    + \| A_2 - \tilde A_2 \|_{\dHe} \les C(\| \psi^{-} \|_{L^2},\| \tilde \psi^{-} \|_{L^2}) (\| \psi^- - \tilde \psi^-
    \|_{L^2}+ \epsilon)
  \end{equation}
  
  iv) if $\psi^- \in L^p$ with $1 \leq p < \infty$ then $\psi^+,
  \frac{\psi_2}r,\frac{A_2-m}r \in L^p$ and
  \begin{equation} \label{strfix} \| \psi^+ \|_{L^p} + \|
    \frac{\psi_2}r \|_{L^p} + \| \frac{A_2-m}r \|_{L^p} 
    \les C(\| \psi^{-} \|_{L^2})\| \psi^-\|_{L^p}
  \end{equation}
  
  v) if $R_{m-1} \psi^- \in H^s$ then $R_{m+1} \psi^+ \in H^s$ for any $s
  \in \{ 1,2,3 \}$, and 
 \begin{equation} \label{sobpsi}
\| R_{m-1} \psi^- \|_{ H^s} \approx \| R_{m+1} \psi^+ \|_{ H^s}
 \end{equation}
with implicit constants depending on $\| \psi^-\|_{L^2}$.
\end{lema}

The reason for having the second type of statement in \eqref{2estdif}
is of technical nature and will be apparent in Section
\ref{concomp}. The equation for $\tilde A_2$ in \eqref{sys2ap} is more
convenient in that form when taking differences. For the original
system \eqref{comp1} it does not matter how one writes the equation
for $A_2$ thanks to the conservation law $A_2^2-|\psi_2|^2= m^2$;
however in the case of \eqref{sys2ap} this conservation law does not
hold true, hence we write the system in the more convenient form
\eqref{sys2ap}.

\begin{proof}
  Our strategy is to solve the ode system \eqref{comp1} from zero.
  Since $\psi_2, A_2-m \in \dHe$, it follows that $\lim_{r
    \rightarrow 0} \psi_2=0, \lim_{r \rightarrow 0}
  A_2=m$. These two conditions play the role of boundary conditions
  at zero. Since $\partial_r (A_2^2-|\psi_2|^2) = 0$, it follows
  from the conditions at $\infty$ that $A_2^2-|\psi_2|^2=m^2$ holds on
  all of $\R_+$.  
    
  To prove existence, we begin by solving the system in a neighborhood
  $(0,R^{-1})$ of the origin. By choosing $R$ large enough we can
  assume without any restriction in generality that
  \begin{equation}\label{psismall}
    \| \psi^-\|_{L^2(0,R^{-1})}  \leq \epsilon
  \end{equation}
  and seek $(\psi_2, A_2)$ with the property that
  \begin{equation}\label{psi2small}
    \|\psi_2\|_{\dot H^1_e(0,R^{-1})} \lesssim \epsilon
  \end{equation}
  Since $\lim_{r \rightarrow \infty} A_2=m$, $A_2^2=m^2+|\psi_2|^2 > 0$  and $A_2$ is continuous, 
  it follows that $A_2= \sqrt{m^2+|\psi_2|^2}$.
  We substitute this in the $\psi_2$ equation and
  discard the dependent $A_2$ equation. We rewrite the $\psi_2$
  equation as
  \[
  (\partial_r + \frac{m}{r}) \psi_2 = i m \psi^- + i (A_2-m) \psi^- 
  - \frac{(A_2-m)\psi_2}{r}
  \]
  or equivalently
  \[
  r^{-m} \partial_r r^m \psi_2 = i m \psi^-+ i(A_2-m) \psi^- -
  \frac{(A_2-m)\psi_2}{r}
  \]
  and further
  \[
  \psi_2 = im r^{-m} [r^{-m} \partial_r]^{-1} \psi^-+ r^{-m} [r^{-m} \partial_r]^{-1}(i
  (A_2-m) \psi^- - \frac{(A_2-m)\psi_2}{r})
  \]
  We know from \eqref{rdrm} that $r^{-m-1}[r^{-m} \bar \partial_r]^{-1}$ maps $L^2$ to
  $L^2$, which easily implies that
  \[
  r^{-m} [r^{-m} \bar \partial_r]^{-1}: L^2 \to \dot H^1_e
  \]
  Hence in order to obtain $\psi_2$ via the contraction principle it
  suffices to show that for $\psi$ as in \eqref{psismall} and $\psi_2$
  as in \eqref{psi2small} the map
  \[
  \psi_2 \to i (A_2-m) \psi^- - \frac{(A_2-m)\psi_2}{r}
  \]
  is Lipschitz from $\dot H^1_e \to L^2$ with a small ($O(\epsilon)$
  in this case) Lipschitz constant. But this is straightforward due to
  the embedding $\dot H^1_e \subset L^\infty$.  Thus the existence of
  $\psi_2$ in $(0,R^{-1}]$ follows, and the corresponding $A_2$ is
  recovered via $A_2(r)=\sqrt{m^2+|\psi_2(r)|^2}$.  The same argument
  also gives Lipschitz dependence of $\psi_2$ on $\psi^-$ in
  $(0,R^{-1}]$.

  The solution obtained above on ($0,R^{-1}]$ can be extended locally
  via standard arguments since $L^2(rdr) \subset L^{1}_{loc}(dr)$.
  This extension is global provided we have an a-priori estimate which
  guarantees that $A_2$ and $\psi_2$ stay in a bounded set.  Indeed,
  integrating the equation of $A_2$ gives
  \[
  A_2(r)-m \leq \| \psi^- \|_{L^2(0,r]} \| \frac{\psi_2}r \|_{L^2(0,r]}  - \| \frac{\psi_2}r \|_{L^2(0,r]}^2 
  = \| \frac{\psi_2}r \|_{L^2(0,r)} (\| \psi^- \|_{L^2(0,r]}  - \| \frac{\psi_2}r \|_{L^2(0,r]} )
  \]
  and since $A_2(r) \geq m$ it follows that $ \| \frac{\psi_2}r
  \|_{L^2(0,r)} \leq \| \psi^- \|_{L^2(0,r]} $ for any $r \geq 0$, in
  particular we obtain $\| \frac{\psi_2}r \|_{L^2} \leq \| \psi^-
  \|_{L^2} $. From above estimate we also obtain
  \begin{equation} \label{A2pest}
  \|A_2\|_{L^\infty} \leq m+\| \psi^- \|_{L^2}^2. 
  \end{equation}
This in turn guarantees that the solution $(A_2,\psi_2)$ extends globally up to $r = \infty$.  
Also, using these estimates in \eqref{comp1} gives the \eqref{2est}.

   For  proving \eqref{loc2est} we use an energy type argument. Denoting
  \[
  F = \frac{\psi_2}{A_2+m}
  \]
  its derivative satisfies
  \[
  \left|\frac{d}{dr} |F|^2 + \frac{2m}r |F|^2\right| \lesssim
  |\psi^-| | F|
  \]
  This further leads to
  \[
  \left|\frac{d}{dr} (r^{2m} |F|) \right| \lesssim r^{2m}|\psi^-|
  \]
  Integrating from infinity we obtain
  \[
  |F| \lesssim r^{-2m} [r^{-2m} \bar \partial_r]^{-1} |\psi^-|
  \]
  Returning to $\psi_2$ we get the pointwise bound
  \begin{equation}\label{psi2point}
    |\frac{\psi_2}{A_2+m}| \lesssim  r^{-2m} [r^{-2m} \bar \partial_r]^{-1} |\psi^-|
  \end{equation}
  Note that if $|\frac{\psi_2}{A_2+m}| \leq \frac{1}{8m}$ then $|\frac{\psi_2}{A_2+m}| \approx |\psi_2|$.
   The construction of the solution on $(0,R^{-1})$ gives the corresponding part of
  \eqref{loc2est} since \eqref{psi2point} holds on any such interval. 
  Getting the $(0,\epsilon^-1 R]$ part of \eqref{loc2est} is slightly more delicate.
  It suffices to get the $L^2$ bound for  $\frac{\psi_2}{r}$.  From \eqref{psi2point} we have
  \[
  |\psi_2| \lesssim r^{-2m} [r^{-2m} \bar \partial_r]^{-1} ( 1_{(0,R]}|\psi^-|) +
  r^{-2m} [r^{-2m} \partial_r]^{-1} (1_{[R,\infty)}|\psi^-|)
  \]
  For the second term we use the smallness of $\psi_2$ in the
  hypothesis.  For the first one we instead produce a pointwise bound
  using Cauchy-Schwarz:
  \[
  r^{-2m} [r^{-2m} \bar \partial_r]^{-1} (1_{(0,R]}|\psi^-|) \lesssim r^{-2m}
  \int_0^R s^{2m} |\psi^-(s)| ds \lesssim (r^{-1} R)^{2m}
  \|\psi^-\|_{L^2}, \qquad r > R
  \]
  This implies the desired $L^2$ bound. 

  Next we turn our attention to \eqref{1estdif} and \eqref{2estdif}.
  In fact, in the case of \eqref{1estdif}, in light of the
  conservation law $\tilde A_2^2 - |\tilde \psi_2|^2=m^2$,
  \eqref{1estdif} follows from \eqref{2estdif} with $E_1=E_2=0$. Hence
  we focus our attention on \eqref{2estdif}. We denote
  \[
  \delta \psi = \tilde{\psi} - \psi, \qquad \delta A_2=\tilde A_2-A_2
  , \qquad \delta \psi_2=\tilde \psi_2 - \psi_2
  \]
  Without any restriction in generality we can make the assumption $\|
  \delta \psi \|_{L^2} \ll 1$ and the bootstrap assumption
  \begin{equation} \label{boot2} \| \delta \psi_2\|_{L^\infty} + \|
    \delta A_2\|_{L^\infty} + \| \frac{\delta \psi_2}r\|_{L^\infty} +
    \|\frac{ \delta A_2}r\|_{L^\infty} \lesssim \epsilon^\frac12 + \|
    \delta \psi \|_{L^2}^\frac12
  \end{equation}

  Then we derive the equations for them modulo error terms.  We have
  \[
  \begin{split}
    \partial_r \delta \psi_2= & \ i \delta A_2 \tilde \psi^- + i A_2
    \delta \psi - \frac1r A_2 \delta \psi_2 - \frac1r \delta A_2
    \tilde \psi_2 + E_1
    \\
    \partial_r \delta A_2 = & - \ \Im ( \psi^- \overline{\delta \psi_2})
    - \Im(\delta \psi \overline{\tilde \psi_2}) - \frac{2}{r} A_2
    \delta A_2 - \frac{1}{r} (\delta A_2)^2+ E_2
  \end{split}
  \]
  The following terms $i A_2 \delta \psi, \Im(\delta \psi
  \overline{\tilde \psi_2}) $ can be directly included into the error
  terms $E_1, E_2$, while the quadratic term $\frac{1}{r} (\delta
  A_2)^2$ can be included in the error term $E_2$ based on
  \eqref{boot2}.  We obtain the following linear system for $(\delta
  \psi_2,\delta A_2)$:
  \[
  \begin{split}
    \partial_r \delta \psi_2= & \ -\frac{m}r \delta \psi_2+ i \tilde
    \psi^- \delta A_2 - \frac1r (A_2-m) \delta \psi_2 - \frac1r \delta
    A_2 \tilde \psi_2 + E_1
    \\
    \partial_r \delta A_2 = & \ - \frac{2m}{r} \delta A_2 + \Im ( \psi^-
    \overline{ \delta \psi_2}) - \frac{2}{r} (A_2-m) \delta A_2 + E_2
  \end{split}
  \]
  By considering the $\Re \delta \psi_2, \Im \delta \psi_2$
  separately, this is a system of the form
  \[
  \partial_r X = - \frac{m}r LX + B X + F, \qquad
  L =\left( \begin{array}{ccc} 1 & 0 & 0 \\
      0 & 1 & 0 \\ 0 & 0 & 2 \end{array} \right)
  \]
  where the matrices $B,F$ satisfy $B \in L^2$ and $F \in L^2+ L^1(dr)$.
  This system needs to be solved with zero Cauchy data at infinity.
  For this system we need to establish the bound
  \begin{equation}
    \| X\|_{L^\infty} +  \| \frac{X}{r}\|_{L^2} \lesssim \| F\|_{L^2 + L^1(dr)}
  \end{equation}
  If $B=0$ then
  \[
  X =\left( \begin{array}{ccc} r^{-m} [r^{-m} \bar \partial_r]^{-1}  & 0 & 0 \\
      0 & r^{-m} [r^{-m} \bar \partial_r]^{-1} & 0 \\ 0 & 0 & r^{-2m}
      [r^{-2m} \partial_r]^{-1} \end{array} \right) F
  \]
  and the conclusion easily follows from argument of type
  \eqref{rdrm}.  If $B$ is small in either $L^2(rdr)$ or in $r^{-1}
  L^\infty$ then we can treat the $BX$ term perturbatively. If $B$ is
  large then some more work is needed.  We decompose $B = B_1+B_2$
  where $B_1 \in L^1(dr)$ and $|B_2| \ll \frac{1}r $.  We can
  construct the bounded matrix $e^{\int B_1}$ as a solution of
  $\partial_r e^{\int B_1}= e^{\int B_1} B_1$ which also has a bounded
  inverse. Then we can eliminate $B_1$ by conjugating with respect to
  $e^{\int B_1}$, and then treat the part with $B_2$ perturbatively.

  iv) From \eqref{psi2point}, \eqref{A2pest} and \eqref{rdrm} we obtain
  \[
  \| \frac{\psi_2}{r} \|_{L^p} \les \| A_2+m \|_{L^\infty}  \| \psi^- \|_{L^p} \les \| \psi^- \|_{L^p} (m+ \| \psi^- \|^2_{L^2})
  \]
  from which \eqref{strfix} follows since $\psi^+= 2i \frac{\psi_2}r +
  \psi^-$ and $A_2-m=\frac{|\psi_2|^2}{A_2+m}$.

  v) Throughout this argument, the use of Sobolev embedding refers to
  the two-dimensional standard Sobolev embeddings which apply to
  $R_{m \pm 1} \psi^\pm$, which then can be read in terms of $\psi^\pm$.

  If $s=1$ then we use \eqref{compnew} to obtain
  \begin{equation} \label{rega1}
  (r \partial_r + (m+1))\psi^+ = (r \partial_r - (m-1)) \psi^- - (A_2-m) (\psi^+ + \psi^-) 
  \end{equation}
  from which
  \[
  \psi^+ = r^{-m-1} [r^m \bar \partial_r]^{-1} \left( (r \partial_r - (m-1)) \psi^- - (A_2-m) (\psi^+ + \psi^-)  \right)  
  \]
  From the Sobolev embedding and \eqref{strfix} we obtain 
  \[
  \| \frac{ (A_2-m) (\psi^+ + \psi^-)}r \|_{L^2} \les  \| \frac{ A_2-m}r\|_{L^4} \| \psi^+ + \psi^- \|_{L^4}  \les \| R_{m-1} \psi^- \|_{H^1}^2
  \] 
  which combined with \eqref{rdrm} gives $\| \frac{\psi^+}r\|_{L^2} \les \| R_{m-1} \psi^- \|_{H^1}$. Plugging this back in 
  \eqref{rega1} gives $\| \partial_{r} \psi^+ \|_{L^2} \les \| R_{m-1} \psi^- \|_{H^1}$ from which the statement follows for $s=1$.

  If $s=2$ we differentiate \eqref{rega1} to obtain
  \[
  \begin{split}
  H_{m+1} \psi^+ & = (\frac1r \partial_r - \frac{m+1}{r^2}) (r \partial_r + (m+1))\psi^+ \\
  & = (\frac1r \partial_r - \frac{m+1}{r^2}) \left[  (r \partial_r - (m-1)) \psi^- - (A_2-m) (\psi^+ + \psi^-)  \right] \\
  &= H_{m-1} \psi^- + (-\frac{2m}r \partial_r + \frac{2m^2-2m}{r^2}) \psi^- 
  - (\frac1r \partial_r - \frac{m+1}{r^2}) \left[  (A_2-m) (\psi^+ + \psi^-)  \right]
  \end{split}	
  \]
   From Lemma \ref{LBE} it follows that $\| (-\frac{2m}r \partial_r + \frac{2m^2-2m}{r^2}) \psi^-  \|_{L^2} \les \| H_{m-1}  \psi^-\|_{L^2}$. 
  From part vii) of Lemma \ref{LBE} we have that $\| \psi^- \|_{L^6} \les \| R_{m-1} \psi^- \|_{H^1}$ and by \eqref{strfix} 
  $\| \frac{\psi_2}r \|_{L^6} + \| \psi^+ \|_{L^6} \les \| R_{m-1} \psi^- \|_{H^1}$, hence we estimate
  \[
  \| \frac1r \partial_r A_2 (\psi^+ + \psi^-)  \|_{L^2} \les (\| \psi^- \|_{L^6} + \| \frac{\psi_2}r \|_{L^6}) 
  \| \frac{\psi_2}r \|_{L^6} (\| \psi^- \|_{L^6} + \| \psi^+ \|_{L^6}) \les \| R_{m-1} \psi^- \|_{H^1}^3
  \]
  \[
  \| \frac{A_2-m}{r^2} (\psi^+ + \psi^-)  \|_{L^2} \les \| \frac{\psi_2}r \|^2_{L^6} 
  (\| \psi^- \|_{L^6} + \| \psi^+ \|_{L^6}) \les \| R_{m-1} \psi^- \|_{H^1}^3
  \]
 Using Lemma \ref{LBE} we estimate
 \[
 \| \frac{A_2-m}r \partial_r \psi^- \|_{L^2} \les  \| \frac{A_2-m}r \|_{L^4}  \| \partial_r \psi^- \|_{L^4}   \les \| R_{m-1} \psi^- \|_{H^2}^3
 \] 
  
 If $s=3$ then from the above expression for $H_{m+1} \psi^+$ we obtain
 \[
  \begin{split}
  \frac1r H_{m+1} \psi^+ & = \frac1r  H_{m-1} \psi^- + \frac1r (-\frac{2m}r \partial_r + \frac{2m^2-2m}{r^2}) \psi^- 
  - (\frac1{r^2} \partial_r - \frac{m+1}{r^3}) \left[  (A_2-m) (\psi^+ + \psi^-)  \right]
  \end{split}	
  \]
 If $m=1$ then the linear part becomes $(\frac{\partial_r^2}{r}-\frac{\partial_r}{r^2})\psi^- \in L^2$ by Lemma \ref{LBE}.
 If $m=2$ then we have $\frac{1}r H_1 \psi^- \in L^2$, and from Lemma \ref{LBE}, it follows that  
 $4(-\frac{\partial_r}{r^2} + \frac{1}{r^3}) \psi^- \in L^2$. If $m \geq 3$, then all the linear terms belong to $L^2$
 in light of Lemma \ref{LBE}. 
 As for the nonlinear terms, we have 
 \[
 (\frac1{r^2} \partial_r - \frac{2}{r^3}) \left[  (A_2-1) (\psi^+ + \psi^-)  \right] 
 = \partial_r \left[  \frac{A_2-1}{r^2} (\psi^+ + \psi^-)  \right] 
 = \frac14 \partial_r \left[  \frac1{A_2+1} |\psi^+ + \psi^-|^2 (\psi^+ + \psi^-)  \right]
 \]
 which can be easily shown to belong to $L^2$ by using vii) and viii) of Lemma \ref{LBE}.

 Finally we apply $\partial_r$ to the expression giving $H_{m+1} \psi^+$ and show that 
 $\partial_r H_{m+1} \psi^+ \in L^2$ in a similar manner. The details are left to the reader.

\end{proof}

\begin{proof}[Proof of Proposition \ref{recprop}]
  With $\psi_2,A_2$ constructed above, we can reconstruct
  $\psi_1=\psi^- +i\frac{\psi_2}r$. Then we solve the system
  \eqref{return} at the level of $(\bar{u}, \bar{v}, \bar{w})$. We
  would like to solve this system with condition at $\infty$,
  $\bar{u}=\overrightarrow{k}, \bar{v}=\overrightarrow{i},
  \bar{w}=\overrightarrow{j}$.  But this cannot be done
  apriori. Indeed, consider the coefficient matrix in \eqref{return}
  \[
  M=\left( \begin{array}{ccc} 0 & \Re{\psi_1} & \Im{\psi_1} \\
      \Re{\psi_1} & 0 & 0 \\ \Im{\psi_1} & 0 & 0 \end{array} \right)
  \]
  Since $M \notin L^1(dr)$, it is not meaningful to initialize the
  problem \eqref{return} at $\infty$. However $M$ has another structure
  which is a consequence of \eqref{comp} rewritten as $\psi_1= (A_2+1) \psi_1 + i \partial_r \psi_2$.
 Therefore $ M= N+ \partial_r K$ and, by \eqref{2est}, $N,K$ satisfy  
 \[
 \| N \|_{L^1(dr)} + \| K \|_{\dHe} \les \| \psi^- \|_{L^2}
 \]
This inequality localizes on intervals $[r,\infty)$ due to \eqref{loc2est}. 
 This allows us to construct solutions with data at $r=\infty$ by using the iteration
  scheme
  \[
  X=\sum_{i} X_i, \qquad X_0=X(\infty), \qquad X_i(r) =
  \int_{r}^\infty M(s) X_{i-1} ds
  \]
  We run the iteration scheme in the space $C([r,\infty])$ of continuous functions on $(r,\infty)$ which
  have limits at $\infty$. Under the assumption that $X_{i-1} \in C([r,\infty])$ we obtain
\[
\begin{split}
X_i(r)& =\int_{r}^\infty (N(s)+\partial_s K(s)) X_{i-1} ds \\
& = \int_{r}^\infty N(s) X_{i-1} ds- K(r)X_{i-1}(r) - \int_{r}^\infty K(s) \partial_s X_{i-1}(s) ds
\end{split}
\]
and further that
\[
\| \partial_r X_i \|_{L^2([r,\infty))} + \| X_i \|_{C([r,\infty])} 
\les \| \psi^- \|_{L^2([r,\infty))} (\| X_{i-1} \|_{L^\infty([r,\infty])}+ \| \partial_r X_{i-1} \|_{L^2([r,\infty))})
\]
 Therefore, inductively, we obtain
 \[
\| \partial_r X_i \|_{L^2([r,\infty))} + \| X_i \|_{C([r,\infty])} 
\les \| \psi^- \|_{L^2([r,\infty))}^i 
\]
  By choosing $R$ large such that $\| \psi \|_{L^2([R,\infty))}$ is small, 
  we can rely on an iteration scheme to construct the solution $X$
  on $[R,\infty)$. 
  
  The uniqueness of this solution is guaranteed by the conservation
  law $|\bar u|,|\bar v|,|\bar w| = constant$ which follows from the particular form of
  $M$.

  This also guarantees that the orthonormality conditions imposed at
  $\infty$ are preserved (recall that $\infty$,
  $\bar{u}=\overrightarrow{k}, \bar{v}=\overrightarrow{i},
  \bar{w}=\overrightarrow{j}$).  The solution constructed above can be
  extended to $(0,\infty)$ by running a similar argument on intervals where $\| \psi^- \|_{L^2(I)}$
  is small, where the last interval is of the form $(0,r]$. 
  
  The above argument leads to an estimate of the form
  \[
  \| X - X_0\|_{C([0,\infty])} + \| \partial_r X \|_{L^2} \les \| \psi^- \|_{L^2}
  \]
  where by $C([0,\infty])$ we mean continuous functions on $(0,\infty)$ which have limits at $0$ and $\infty$. 
  
  Additional information on $\bar{u},\bar{v},\bar{w}$ will be obtained
  in a different manner. Notice that $\bar{u}_3$ and $\zeta =
  \bar{w}_3- i \bar{v}_3$ solve the system
  \[
  \partial_r \bar{u}_3 = -\Im{(\psi_1 \bar{\zeta})}, \qquad \partial_r
  \zeta = i \bar{u}_3 \psi_1
  \]
  which is the same as the one satisfied by $A_2, \psi_2$. Since the
  conditions at $\infty$ are proportional with a constant $m$,
  we conclude that $m \bar{u}_3=A_2, - m \zeta = \psi_2$.
  From this and the fact that $A_2^2 - |\psi_2|^2=m^2$ it follows also
  that $m^2(|\bar u_1|^2 + |\bar u_2|^2)=|\psi_2|^2$. 
  
  Next, we extend the system of vectors to $u,v,w$ using the
  equivariant setup, i.e.  by multiplying them with $e^{m \theta
    R}$. Using the identification just described above and the
  orthonormality conditions, it follows that \eqref{return} is
  satisfied for $k=2$.  Therefore we have just established the
  existence of an equivariant map $u$ whose vector field
  $\mathcal{W}^-$ in the gauge $(v,w)$ is $\psi^-$ and whose gauge
  elements are $\psi_1,\psi_2,A_2$. Moreover, we have that
  \[
  E(u)= \pi \| \psi^- \|_{L^2}
  \]

Given two fields $\psi^-, \tilde \psi^-$ we reconstruct $X$ and $\tilde X$ as above. Since the construction is
iterative it also follows that
\[
\| X-\tilde X \|_{C[0,\infty]} + \| \partial_r (X-\tilde X) \|_{L^2} \les \| \psi - \tilde \psi \|_{L^2}
\] 
from which the derivative part in $E(u-\tilde u)$ follows. Since $u_1 = v_2 w_3 - v_3 w_2$, 
$\tilde u_1 = \tilde v_2 \tilde w_3 - \tilde v_3 \tilde w_2$, $\psi_2 = -m(\bar w_3 - i \bar v_3)$ and
$\tilde \psi_2 = -m(\bar{ \tilde w}_3 - i \bar{ \tilde v}_3)$ it follows that
\[
\|\frac{u_1 - \tilde u_1}r \|_{L^2} 
\les \| \frac{\psi_2-\tilde \psi_2}r \|_{L^2} \| X \|_{L^\infty} + \| X - \tilde X \|_{L^\infty} \| \frac{\tilde \psi_2}r \|_{L^2}
\les \| \psi^- - \tilde \psi^- \|_{L^2}
\]
A similar argument shows that $\|\frac{u_1 - \tilde u_1}r \|_{L^2}  \les \| \psi^- - \tilde \psi^- \|_{L^2}$ which completes 
the proof of \eqref{udif}. 
 
\end{proof}

\section{The Cauchy problem} \label{CTH}

In this section we are concerned with the nonlinear system of
equations \eqref{psieq} which we recall here
\begin{equation*}
 \left\{ \begin{array}{l} 
     (i \partial_t + H_{m-1}) \psi^- = (A_0 - 2 \frac{A_2 - m}{r^2} +
      \frac1{r}\Im{(\psi_2 \bar{\psi}^-)}) \psi^- \cr 
      (i \partial_t + H_{m+1}) \psi^+ = ( A_0 + 2 \frac{A_2 - m}{r^2} - \frac1{r}\Im{(\psi_2
        \bar{\psi}^+)}) \psi^+
    \end{array} \right.
\end{equation*}
where $\psi_2,A_2,A_0$ are given by \eqref{rela}, \eqref{A2},
respectively \eqref{A0}.  The problem comes with an initial data
$\psi^\pm(t_0)=\psi^\pm_0$ and we would like to understand its
well-posedness on intervals $I \subset \R$ with $t_0 \in I$.

We will be mainly interested in solutions of this system which come
from Schr\"odinger maps, i.e. they satisfy the compatibility conditions \eqref{compnew}.

For simplicity we denote the nonlinearities by
\begin{equation} \label{psin}
  \begin{split}
    N^\pm_m(\psi^\pm) = (A_0 \pm 2 \frac{A_2-m}{r^2} \mp
    \frac1{r}\Im{(\psi_2 \bar{\psi}^\pm)}) \psi^\pm 
   \end{split}
\end{equation}
 
We define the mass of a function $f$ by $M(f):=\| f \|^2_{L^2}$.  The
system \eqref{psieq} formally conserves the mass,
i.e. $M(\psi^-(t))=M(\psi^-(0))$ and $M(\psi^+(t))=M(\psi^+(0))$ for
all $t$ in the interval of existence. Moreover, as discussed in
subsection \ref{COG}, a compatible pair also satisfies $\| \psi^+(0)
\|_{L^2}=\| \psi^-(0) \|_{L^2}$.

\subsection{Strichartz estimates}

We begin our analysis by understanding the linear equation
\begin{equation} \label{beq} (i \partial_t + H_k) u =f, \qquad
  u(0)=u_0
\end{equation}
where we recall $H_k=\partial_r^2 + \frac1r \partial_r -
\frac{k^2}{r^2}$. 

Our first claim is that, for each $k$, $u$ satisfies the standard
Strichartz estimates
\begin{equation} \label{STR} \| |\nabla|^s R_k u \|_{L^p_t L^q_r} \les
  \||\nabla|^s R_k u_0 \| + \| |\nabla|^s R_k f \|_{L^{\tilde p'}_t
    L^{\tilde q'}_r}
\end{equation}
where $|\nabla|^s=(-\Delta)^\frac{s}2$ (defined in the usual manner),
$(p,q),(\tilde p, \tilde q)$ are admissible pairs in two dimensions
($\frac{1}{p}+\frac{1}{q}=\frac12, 2 < p \leq \infty$) and $(\tilde
p', \tilde q')$ is the dual pair of $(\tilde p, \tilde q)$. Indeed,
$R_k u$ satisfies the following equation
\[
\begin{split}
  (i\partial_t + \Delta) R_k u= R_k f, \qquad R_k u(0)=R_k u_0
\end{split}
\]
Then the Strichartz estimates follow from the standard Strichartz in
two dimensions.  We need to read the Strichartz estimates
at the level of the radial functions.  For even powers of $s$ we use
the identity $\Delta R_k v = R_k H_k v$, hence
\begin{equation} \label{STRE} \| H_k v \|_{L^p_t L^q_r} = \| \Delta
  R_k v \|_{L^p_t L^q_x}
\end{equation}
and this can be extended to higher regularity but we will not need it.

For odd values of $s$ we use that $|\nabla|^s=|\nabla|
(-\Delta)^{\frac{s-1}2}$ and that for $k \ne 0$
\begin{equation} \label{STRO} \| \partial_r v \|_{L^p_t L^q_r} + \|
  \frac{v}r \|_{L^p_t L^q_r} \les \| |\nabla| R_k v \|_{L^p_tL^q_x}
\end{equation}
while for $k=0$
\begin{equation} \label{STRZ} \| \partial_r v \|_{L^p_t L^q_r} \les \|
  |\nabla| R_k v \|_{L^p_tL^q_x}
\end{equation}

In the context of additional regularity, we need to make improved
versions of the Strichartz estimates. We recall the following result from \cite{BIKT2}. 

\begin{lema} \label{LB} Assume that $u$ satisfy \eqref{beq} with
  initial data $u_0$ and forcing $f$.
  
  i) If $u_0 \in L^2$ is such that $H_k u_0 \in L^2$, for $k \geq 2$, then the
  following holds true
  \[
  \||\partial_r^2 u| + | \frac{\partial_r u}{r}|+
  |\frac{u}{r^2}| \|_{L^\infty L^2 \cap L^4 L^4 \cap L^3L^6} \les \|
  H_k u_0 \|_{L^2} + \| H_k f \|_{L^1L^2}
  \]
  ii) If $u_0 \in L^2$ is such that $H_1 u_0 \in L^2$  then the
  following holds true
  \[
 \|  \partial_r^2 u \|_{L^\infty L^2  \cap L^4 L^4 \cap L^3L^6} + \| \frac{1}{r} (\partial_r-\frac1r) u \|_{L^\infty L^2
    \cap L^4 L^4 \cap L^3L^6} \les \| H_1 u_0 \|_{L^2} + \| H_1 f \|_{L^1 L^2}
  \]
\end{lema}
These are improved versions of Strichartz estimates from the following
point of view.  In i) the inequality for $(\partial_r^2 +
\frac1r \partial_r - \frac{k^2}{r^2}) u=H_2 u$ is the Strichartz
estimate for $H_2 u$ which follows from \eqref{STR} and \eqref{STRE};
our statement is stronger in saying that each term satisfies the
Strichartz estimate. A similar remark is in place for part i). Note the consistency with Lemma \ref{LBE}. 

\subsection{Setup and Cauchy theory}
In order to make estimates shorter, we make the following notation
convention $\| f^\pm \| = \| f^+ \| + \| f^- \|$ for various $f$'s and
$\| \cdot \|$ involved in the rest of the paper.

Since our non-linear analysis relies mostly on the $L^4_{t,r}$ norm,
we define the Strichartz norm of $f:I \times \R^2 \rightarrow \C$ by
$S_I(f):= \| f \|^4_{L^4(I \times \R)}$.  If $t_0 \in I$ then we
define $S_{I,\leq t_0} f = \| 1_{I \cap (-\infty,t_0]} f \|^4_{L^4}$
and $S_{I,\geq t_0} f = \| 1_{I \cap [t_0, \infty)} f \|^4_{L^4}$.

We say that a solution $\psi^\pm: I \times \R \rightarrow \C$ blows up
forward in time if $S_{I,\geq t} \psi^\pm = + \infty, \forall t \in
I$.  Similarly $\psi^\pm$ blows up backward in time if $S_{I,\leq t}
\psi^\pm = + \infty, \forall t \in I$.

A possibility that may occur is that for some interval $I$, $S_{I,\geq
  t_0} \psi^+ = +\infty$ while $S_{I,\geq t_0} \psi^- < \infty$, or
any other combination. However from \eqref{str} it follows that
solutions satisfying the compatibility condition \eqref{compnew} we
have that $S_J(\psi^+) \approx S_J(\psi^-)$ on any time interval
$J$. Therefore for such solutions (which we will be mainly interested
in) the above scenario is ruled out.

Let $\psi_+^\pm \in L^2$. We say that the solution $\psi^\pm: I \times
\R \rightarrow \C$ scatters forward in time to $\psi_+^\pm$ iff $\sup
I= + \infty$ and $\lim_{t \rightarrow \infty}
M(\psi^\pm(t)-e^{itH_{m \pm 1}}\psi_+^\pm)=0$.  We say that the solution
$\psi^\pm: I \times \R \rightarrow \C$ scatters backward in time to
$\psi_{-}^\pm$ iff $\inf I= - \infty$ and $\lim_{t \rightarrow -\infty}
M(\psi^\pm(t)-e^{itH_{m \pm 1}}\psi_-^\pm)=0$.

Our first theorem provides the general Cauchy theory for
\eqref{psieq}.

\begin{theo} \label{CT} Consider the problem \eqref{psieq} (with
  $\psi_2,A_2,A_0$ given by \eqref{rela}, \eqref{A2}, \eqref{A0}) with
  $\psi_0^\pm \in L^2$. Then there exists a unique maximal-lifespan
  solution pair $(\psi^+,\psi^-): I \times \R^2$ with $t_0 \in I$ and
  $\psi^\pm(t_0)=\psi_0^\pm$ with the additional properties:

  i) I is open.

  ii) (Forward scattering) If $\psi^\pm$ do not blow up forward in
  time, then $I_+=[0,\infty)$ and $\psi^\pm$ scatters forward in time
  to $e^{itH_{m \pm 1}} \psi_+^\pm$ for some $\psi_+^\pm \in L^2$.

  Conversely, if $\psi_+^\pm \in L^2$, then there exists a unique
  maximal-lifespan solution $\psi^\pm$ which scatters forward in time
  to $e^{itH_{m \pm 1}} \psi_+^\pm$.

  iii) (Backward scattering) A similar statement to ii) holds true for
  the backward in time problem.

  iv) (Small data scattering) There exist $\epsilon > 0$ such that if
  $M(\psi_0^\pm) \leq \epsilon$ then $S_{\R}(\psi^\pm) \les
  M(\psi_0^\pm)^2$. In particular, the solution does not blow up and
  we have global existence and scattering in both directions.

  v) (Uniformly continuous dependence) For every $A > 0$ and $\epsilon
  > 0$ there is $\delta > 0$ such that if $\psi^\pm$ is a solution
  satisfying $S_J(\psi^\pm) \leq A$ and $t_0 \in J$, and such that
  $M(\psi_0^\pm - \tilde \psi_0^\pm) \leq \delta$, then there exists a
  solution such that $S(\psi^\pm -\tilde \psi^\pm) \leq \epsilon$ and
  $M(\psi(t)-\tilde \psi(t)) \leq \epsilon, \forall t \in J$.

  vi) (Stability result) For every $A > 0$ and $\epsilon > 0$ there
  exists $\delta > 0$ such that if $S_J(\psi^\pm) \leq A$, $\psi^\pm,$
  approximate \eqref{psieq}, in the sense
  \[
  \| (i \partial_t + H_{m \pm 1}) \psi^\pm - N^\pm(\psi^\pm) \|_{L^\frac43(J
    \times \R)} \leq \delta,
  \]
  $t_0 \in J, \tilde \psi_0^\pm \in L^2$ and
  $S_J(e^{i(t-t_0)H^\pm}(\psi^\pm(t_0) - \tilde \psi_0^\pm)) \leq
  \delta$, then there exists a solution $\tilde \psi^\pm$ on $I$ to
  \eqref{psieq} with $\tilde \psi^\pm(t_0)= \tilde \psi_0^\pm$ and
  $S_J(\psi^\pm - \tilde \psi^\pm) \leq \epsilon$.

  vii) (Additional regularity) Assume that, in addition, $R_{m \pm 1}
  \psi^\pm_0 \in H^s$ for $s \in \{1,2,3\}$. If
  $J$ is an interval such that $S_{J}(\psi^\pm) \leq A < + \infty$,
  then the solution $\psi^\pm$ satisfies
  \begin{equation} \label{addreg} \| R_{m \pm 1} \psi^\pm(t) \|_{H^s} \les_A
    \| R_{m \pm 1} \psi^\pm_0 \|_{H^s}, \qquad \forall t \in J
  \end{equation}  
  and it also has Lipschitz dependence with respect to the initial
  data.
\end{theo}
The above results are concerned with general solutions of
\eqref{psieq}. However, our interest lies in solutions which
correspond to geometric maps.  The next result completes the Cauchy
theory for solutions of \eqref{psieq} which satisfy the compatibility
condition \eqref{compnew}.  The system \eqref{psieq} does not directly
involve the variable $\psi_0$ which is defined in this context by
\eqref{psizero}.

\begin{theo} \label{CT-CC} i) If $\psi^\pm_0 \in L^2$ satisfying the
  compatibility condition \eqref{compnew}, then $\psi^\pm(t)$
  satisfies the compatibility condition \eqref{compnew} for each $t
  \in I$. If, in addition, $R_{m \pm 1} \psi^\pm_0 \in H^3$ then
  \eqref{compat} and \eqref{curb} are satisfied.
  
  ii) If the solution satisfies the compatibility condition
  \eqref{compnew} and it does not blow up in time then the two
  scattering states (described in ii)) are related by
  \begin{equation} \label{asscat}
    \partial_r r (\psi_+^+ - \psi_+^-) = -m(\psi_+^+ + \psi_+^-)
  \end{equation}  
  Conversely, if $\psi_+^\pm \in L^2$ satisfy \eqref{asscat}, then the
  unique maximal-lifespan solution $\psi^\pm$ which scatters to
  $e^{itH^\pm} \psi_+^\pm$ (constructed in part ii)) satisfy the
  compatibility condition \eqref{compnew}.  A similar statement holds
  true for the backward in time scattering.

  iii) If $\psi^\pm$ satisfy the compatibility conditions, then for
  every interval $J \subset I$ (I being the maximal-lifespan interval)
  the following holds true
  \begin{equation} \label{str} \| \psi^+ \|_{L^4(J)} \approx \| \psi^-
    \|_{L^4(J)}
  \end{equation}
  where the constants involved in the use $\approx$ are independent of
  the interval $J$.

\end{theo}

As a consequence of these theorems we are able to prove the following
result
\begin{prop} \label{SMR} If $\psi^\pm_0 \in L^2$ satisfies the
  compatibility conditions \eqref{compnew}, $R_{m \pm 1} \psi^\pm_0 \in H^2$
  and $\psi^\pm(t)$ is the solution of \eqref{psieq} on $I$ then the
  map $u(t)$ constructed in Proposition \ref{recprop} (for each $t$)
  is a Schr\"odinger map.
\end{prop}

\begin{proof}[Proof of Theorem \ref{CT}]  
  Parts i)-vi) are standard. Our particular setup is very similar to the one 
  in the Theorem 3.2 in \cite{BIKT2}, and the proof there can be easily adapted 
  to our problem.
  
  As discussed in \cite{BIKT2}, part vii) is usually standard, with the exception
  of one term in it. We rewrite the nonlinear terms as follows
  \[
    A_0 \pm 2 \frac{A_2-m}{r^2} \mp \frac1{r}\Im{(\psi_2 \bar{\psi}^-)}
   = \frac{|\psi^-|^2}2 - [r\partial_r]^{-1} \Re(\bar \psi^+ \psi^-) \pm \frac1{2r^2} \int_0^r (|\psi^+|^2 - |\psi^-|^2)sds 
  \]
  Without the term $[r\partial_r]^{-1} \Re(\bar \psi^+ \psi^-)$, the
  analysis would be standard, see  \cite{BIKT2} for more commentaries.
   We will provide a full analysis of the term
  \[
  N_1^\pm = [r\partial_r]^{-1} \Re(\bar \psi^+ \psi^-) \psi^\pm
  \]
  This analysis can be extended to the other two terms in
  $N^\pm(\psi^\pm)$.

  The analysis in the case $m=1$ is similar to the one in \cite{BIKT2}.
  We now proceed with the cases $m \geq 2$.
  Since $S_I(\psi^\pm) \leq A$, the standard theory gives also that
  \[
  \| \psi^\pm \|_{L^3L^6(I \times \R)} \les_A 1
  \]
  Therefore it makes sense to define
  \[
  \begin{split}
    & B=\| \partial_r \psi^\pm \|_{L^3L^6} +   \| \frac{\psi^\pm}r \|_{L^3L^6}  \\
    &C = \| \partial_r^2 \psi^\pm \|_{L^3 L^6} + \| \frac1r \partial_r
    \psi^+ \|_{L^3 L^6}+  \| \frac{\psi^+}{r^2}  \|_{L^3 L^6} 
    +  \| \frac1r (\partial_r-\frac1r) \psi^- \|_{L^3L^6} \\
    & D = \| \partial_r H^\pm \psi^\pm \|_{L^3L^6} + \| \frac1r H^\pm
    \psi^\pm \|_{L^3L^6} 
  \end{split}
  \]
  We will prove the following estimates
  \begin{equation} \label{regineq}
    \begin{split}
      & \| \partial_r  N_1^\pm \|_{L^1L^2} + \| \frac1r N_1^\pm \|_{L^1 L^2}  \les_A B \\
      & \| H_{m \pm 1} N_1^\pm \|_{L^1 L^2} \les_{A}  C + B^2 \\
      & \| \partial_r H_{m \pm 1} N_1^\pm \|_{L^1L^2} + \| \frac1r H_{m \pm 1} N_1^-
      \|_{L^1 L^2} \les_A D + BC
    \end{split}
  \end{equation}
  Similar estimates hold true for the other two terms
  in $N^\pm(\psi^\pm)$. Based on these estimates, the Strichartz estimates \ref{STR} 
  and the result of Lemma \ref{LB}, a standard argument establishes the conclusion in
  \eqref{addreg}.

  We now turn to the proof of \eqref{regineq}. We compute
  \[
  \partial_r N_1^\pm = \partial_r \left( [r\partial_r]^{-1} \Re(\bar
    \psi^+ \psi^-) \right) \psi^\pm + [r\partial_r]^{-1} \Re(\bar
  \psi^+ \psi^-) \partial_r \psi^\pm
  \]
  and estimate
  \[
  \| \partial_r N_1^\pm \|_{L^1L^2} \les \| \psi^+ \|_{L^3L^6} \|
  \frac{ \psi^-}r \|_{L^3 L^6} \| \psi^\pm \|_{L^3L^6} + \| \psi^\pm
  \|^2_{L^3L^6} \| \partial_r \psi^\pm \|_{L^3L^6}
  \]
  from which half of the first estimate in \eqref{regineq} follows;
  the second half follows in a similar manner.

  We continue with
  \[
  \begin{split}
    H_{m \pm 1}N_1^\pm = & \Delta \left( [r\partial_r]^{-1} \Re(\bar \psi^+
      \psi^-) \right) \psi^\pm
    + 2 \partial_r \left( [r\partial_r]^{-1} \Re(\bar \psi^+ \psi^-) \right) \partial_r \psi^\pm \\
    & + \left( [r\partial_r]^{-1} \Re(\bar \psi^+ \psi^-) \right)
    H_{m \pm 1} \psi^\pm
  \end{split}
  \]
  The last term is estimated by $\les_A C$, the second one is
  estimated by $\les_A B^2$, while the first one equals
  \[
  (\partial_r + \frac1r) \frac{\Re(\bar \psi^+ \psi^-)}r \cdot
  \psi^\pm = \frac{\Re(\partial_r \bar \psi^+ \cdot \psi^-)+ \Re(\bar
    \psi^+ \cdot \partial_r \psi^-)}r \cdot \psi^\pm
  \]
  and its $L^1L^2$ norm is estimated by
  \[
  \les ( \| \partial_r \psi^+ \|_{L^3L^6} \| \frac{\psi^-}r
  \|_{L^3L^6}+ \| \frac{\psi^+}r \|_{L^3L^6} \| \partial_r \psi^-
  \|_{L^3L^6} ) \| \psi^\pm \|_{L^3L^6}
  \]
  from which the second estimate in \eqref{regineq} follows.

  For the third estimate we start with
  \[
  \begin{split}
    \partial_r H_{m \pm 1} N_1^\pm & = \partial_r \Delta \left(
      [r\partial_r]^{-1} \Re(\bar \psi^+ \psi^-) \right) \psi^\pm
    + \Delta \left( [r\partial_r]^{-1} \Re(\bar \psi^+ \psi^-) \right) \partial_r \psi^\pm \\
    & + 2 \partial_r^2 \left( [r\partial_r]^{-1} \Re(\bar \psi^+
      \psi^-) \right) \partial_r \psi^\pm
    + 2  \partial_r \left( [r\partial_r]^{-1} \Re(\bar \psi^+ \psi^-) \right) \partial_r^2 \psi^\pm \\
    & + \partial_r \left( [r\partial_r]^{-1} \Re(\bar \psi^+ \psi^-)
    \right) H_{m \pm 1} \psi^\pm + \left( [r\partial_r]^{-1} \Re(\bar \psi^+
      \psi^-) \right) \partial_r H_{m \pm 1} \psi^\pm
  \end{split}
  \]
  The $L^1L^2$ norm of the sixth terms above is bounded by $\les_A
  D$. Using the previous arguments, the $L^1L^2$ norm of the second,
  fourth and fifth term is bounded by $\les_A BC$.  Since
  \[
  \begin{split}
    \partial_r^2 \left( [r\partial_r]^{-1} \Re(\bar \psi^+ \psi^-)
    \right) = \frac{\Re(\partial_r \bar \psi^+ \cdot \psi^-)}r 
    + \Re(\frac{\bar \psi^+}r( \partial_r  -\frac1r)\psi^-)
  \end{split}
  \]
  it follows that the $L^1L^2$ norm of the third term above is bounded
  by $\les_A BC$.

  The first terms is further expanded 
  \[
  \begin{split}
    & \partial_r \Delta \left( [r\partial_r]^{-1} \Re(\bar \psi^+
      \psi^-) \right)
    = \partial_r \left( \frac{\Re(\partial_r \bar \psi^+ \cdot \psi^-)+ \Re(\bar \psi^+ \cdot \partial_r \psi^-)}r \right) \\
    & = \frac{\Re(\partial_r^2 \bar \psi^+ \cdot \psi^-)+ 2
      \Re(\partial_r \bar \psi^+ \cdot \partial_r \psi^-)}r + \Re(\frac{\bar \psi^+}r (\partial_r - \frac1r) \partial_r \psi^-)
  \end{split}
  \]
  and estimated by $BC$. The estimate for $\frac1r H_{m \pm 1} N_1^\pm$ is obtained along the same lines,
  though the argument is much easier. The details are left to the reader. This  finishes the argument for \eqref{regineq}.
\end{proof}

\begin{proof}[Proof of Theorem \ref{CT-CC}]
  i) The proof follows exactly the same steps as in \cite{BIKT2}, with the only adjustments coming
  from the value of $\mu=-1$ and that we work with a general $m$. 
    
  It is useful to rephrase this in terms of $\psi_1$, $\psi_2$,
  which are recovered linearly from $\psi^{\pm}$. Reverting the
  algebraic computation from Sections \ref{CG} and \ref{COG},
  $\psi_1$, $\psi_2$ solve the system \eqref{dtpsiab}. Then we seek to
  show that the relation $D_1 \psi_2 = D_2 \psi_1$ is preserved along
  the flow. For this we will derive an equation for the quantity
  \[
  F = D_2 \psi_1 - D_1 \psi_2
  \]
   Following the lines of the argument in \cite{BIKT2} we derive the following equation for $F$:
  \[
  i D_0 F = (\frac{A_2^2}{r^2} - \partial_1 (\partial_1+\frac{1}r))F +
  \Re (F \bar \psi_1) \psi_1 - \frac{1}{r^2} \Re ( F \bar
  \psi_2)\psi_2
  \]
  It is more convenient to recast this as an equation for
  \[
  \frac{F}{r} = - (\partial_r + \frac1r)\frac{\psi_2}{r} + \frac{i A_2}{r} \psi_1
  \]
  which is exactly the quantity in \eqref{compnew}. We obtain
  \begin{equation} \label{Freq} 
  (i \partial_t +H_m) \frac{F}{r} = 
  (A_0+\frac{A_2^2-m^2}{r^2}) F + \Re (\frac{F}{r} \bar
    \psi_1) \psi_1 - \frac{1}{r^2} \Re ( \frac{F}{r} \bar \psi_2)\psi_2
  \end{equation}
  In view of the $L^4$ Strichartz bounds for $\psi_1$ and $\psi_2$ and
  the derived $L^2$ bounds for $A_0$ and $\frac{A_2^2-m^2}{r^2}$,
  standard arguments show that this linear equation is well-posed in
  $L^2$. Hence the conclusion follows provided that $\frac{F}{r}$ has
  sufficient regularity. Indeed, we have
  \[
  \frac{F}r = \frac{i}2 \left( \partial_r \psi^+ + \frac{1+A_2}r
    \psi^+ - \partial_r \psi^- - \frac{1-A_2}r \psi^- \right)
  \]
  It is obvious that if $R_{m \pm 1} \psi^\pm \in H^1$ then $\frac{F}r \in
  L^2$.
  
  If $R_{m \pm 1} \psi^\pm \in H^2$ then by using the results in Lemma
  \ref{LBE} and Sobolev embeddings one easily shows that $\frac{F}r
  \in \dHe$.
  
  We will show in detail that if $R_{m \pm 1} \psi^\pm \in H^3$, then $ H_m \frac{F}r \in L^2$. Indeed,
  \[
  \begin{split}
    -2i H_m \frac{F}r &  = (\partial_r + \frac{1+A_2}{r})H_{m+1} \psi^+ + 2(m+1) \frac{A_2-m}{r^3} \psi^+ 
      + 2 \frac{m-A_2}{r^2} \partial_r \psi^+ \\
      & + \frac{(\partial_r-\frac1r) \partial_r A_2}r \psi^+ + 2 \frac{\partial_r A_2}{r} \partial_r \psi^+
      -(\partial_r + \frac{1-A_2}{r})H_{m-1} \psi^- \\
      & + 2 \frac{m-A_2}{r^2} \partial_r \psi^- + 2(m-1) \frac{m-A_2}{r^3} \psi^- 
      + \frac{(\partial_r-\frac1r) \partial_r A_2}r \psi^- + 2 \frac{\partial_r A_2}{r} \partial_r \psi^-
  \end{split}
  \]
The above expression is easily shown to belong to $L^2$ based on that $R_{m \pm 1} \psi^\pm \in H^3$, 
by using that $R_{m \pm 1} H_{m \pm 1} \psi^\pm \in H^1$, the Sobolev embeddings 
$\psi^\pm, \partial_r \psi^\pm \in L^6$  and \eqref{A2}.

  Hence we can conclude that $H_m \frac{F}r \in
  L^2$.  This allows us to run a standard energy argument by pairing
  the equation, with $\bar F$, to conclude that
  \[
  \partial_t \| \frac{F}r \|^2_{L^2} \les (\| \psi_1 \|^2_{L^\infty} +
  \| \frac{\psi_2}r \|^2_{L^\infty}) \| \frac{F}r \|^2_{L^2}
  \]
  which by using the Gronwall inequality and the fact that $F(0)=0$
  leads to $F(t)=0$ for all $t \in I$.
  
  In order to run the energy argument it suffices to have $\frac{F}r
  \in \dHe$ and use the pairing of $\dot H^{-1}_e$ and $\dHe$. This is
  useful in the proof of Proposition \ref{SMR} where we assume only
  $R_{m \pm 1} \psi^\pm \in H^2$.
  
  In the general case when $\psi^\pm_0 \in L^2$ only we regularize
  them as follows. We produce $R_{m - 1} \psi^-_{n,0} \in H^3$ so that $\|
  \psi^-_0 - \psi^-_{n,0} \|_{L^2} \leq \frac1n$.  By using Lemma
  \eqref{a2p2c}, and particularly part v), we obtain that the
  compatible pair $R_{m +1} \psi^+_{n,0} \in H^3$ and $\| \psi^+_0 -
  \psi^+_{n,0} \|_{L^2} \les \frac1n$. We also recast the
  compatibility condition to
  \[
  \psi^+-\psi^-=- [r\partial_r]^{-1} \left( \psi^+ - \psi^-
    +A_2(\psi^++\psi^-)\right)
  \]
  so that all terms involved belong to $L^2$.  Using the conservation
  of the compatibility condition for $\psi^\pm_n(t)$ under the flow
  \eqref{psieq} and part v) of the Theorem, we obtain the desired
  result.

  ii) The key observation is that the equation for $\psi_2$ in
  \eqref{comp1} becomes linear in the following sense:
  \begin{equation} \label{limcomp} \lim_{t \rightarrow \infty}
    \| \partial_r \psi_2 - i m \psi^- + m\frac{\psi_2}r \|_{L^2} =0
  \end{equation}
  under the hypothesis that $\lim_{t \rightarrow \infty} \| \psi^-(t)
  - e^{itH_{m-1}} \psi^-_+ \|_{L^2}=0$. This is easily shown to follow
  from the following estimate
  \begin{equation} \label{pd} \lim_{t \rightarrow \infty} \sup_{r \in
      (0,\infty)} | r^{-m} \int_0^r e^{itH_{m-1}} f(s) s^m ds| = 0
  \end{equation}
  which holds true for $f \in L^2$. The proof of \eqref{pd} is similar to the corresponding statement in the
  Appendix of \cite{BIKT2}. Based on this, it follows that $\lim_{t \rightarrow
    \infty} \| \psi_2(t) \|_{L^\infty} =0$, and that
  \[
  \lim_{t \rightarrow \infty} \|i (A_2-m) \psi^- - \frac{1}r (A_2-m)
  \psi_2\|_{L^2} =0
  \]
  which justifies \eqref{limcomp}.

  With the notation (essentially the linearized version of $\frac{F}r$ above)
  \[
  f(t) = \partial_r ( e^{itH_{m+1}} \psi_+^+ - e^{itH_{m-1}}\psi_+^-) 
  - 2 \frac{e^{itH_{m-1}}\psi_+^-}r,
  \]
  the scattering relation \eqref{asscat} can be rewritten as $\lim_{t
    \rightarrow \infty} \| f(t) \|_{\dot H_e^{-1}}=0$.  A direct
  computation gives that $f$ obeys the equation
  \[
  \begin{split}
    (i\partial_t + H_m) f = 0
  \end{split}
  \]
  Since $\lim_{t \rightarrow \infty} \| f(t) \|_{\dot H^{-1}_e} =0$ it
  follows from the conservation of the $\dot H^{-1}_e$ norm that
  $f(0)=0$ which is \eqref{asscat}. Alternatively, one could carry out
  this argument as we did in viii).

  Assume now that given $\psi^\pm_+$ satisfying \eqref{asscat} we
  construct (as in ii)) solutions $\psi^\pm(t)$ to \eqref{psieq} on
  some $[T,+\infty)$ which scatter forward to $e^{itH_{m \pm 1}}
  \psi^\pm_+$. Following the argument in part i) we construct $F$
  which satisfies \eqref{Freq}. Assuming additional regularity on the
  states $\psi^\pm_+$, $R_{m \pm 1} \psi^\pm_+ \in H^3$, we have by part
  vii) of Theorem \ref{CT} that $R_{ m \pm 1} \psi^\pm(t) \in H^3$, hence by
  the argument in i) $H_m \frac{F}r \in L^2$ and the
  right-hand side of \eqref{Freq} belongs to $L^2$. Then the Duhamel
  formula applies to \eqref{Freq} and in turn the Strichartz estimate
  \[
  \| \frac{F}r \|_{L^4([T,\infty) \times \R)} \les \| \psi^\pm
  \|_{L^4([T,\infty) \times \R)}^2 \| \frac{F}r \|_{L^4([T,\infty)
    \times \R)}
  \]
  where we have used that $\lim_{t \rightarrow \infty} \| \frac{F(t)}r
  \|_{L^2}=0$ (this follows as above because of \eqref{pd}). Next, by
  taking $T$ large enough, we obtain that $F(t) \equiv 0$ for $t \geq
  T$ and the conclusion follows by invoking part i).

  For general states $\psi^\pm_+ \in L^2$ satisfying \eqref{asscat} we
  proceed as above.  We approximate them by sequences $\psi^\pm_{n,+}$
  with $R_{ m \pm 1} \psi^\pm_{n,+} \in H^3$; this can be done by
  regularizing $R_{m -1} \psi^-_+$ first and then showing that the
  corresponding $R_{m + 1} \psi^+_+$ has the same regularity as we did in
  Lemma \ref{a2p2c} part v) - in fact this argument involves only the
  linear part of the argument there. Then we write \eqref{asscat} at
  the level of $L^2$
  \[
  \psi^+_+ - \psi^-_+=  - [r\partial_r]^{-1} ( (m+1)\psi^+_+ + (m-1) \psi^-_+ ) ,
  \]
  use the above argument and a limiting argument.

  iii) One side of \eqref{str} follows from the fixed time bound
  \eqref{strfix}. The other side is similar, and it consists and
  replicating the result of Lemma \ref{a2p2c} starting from $\psi^+$
  instead.

\end{proof}

\begin{proof}[Proof of Proposition \ref{SMR}]
  With the given $\psi_0^\pm$ we reconstruct $u_0 \in \dot H^1 \cap
  \dot H^3$ as in Proposition \ref{recprop}.  The additional
  regularity $R_{m \pm 1} \psi^\pm_0 \in H^2$ implies, by \eqref{rtr}, that
  $u_0 \in \dot H^1 \cap \dot H^3$.  For the classical Schr\"odinger
  Map $u(t)$ with data $u_0$ we construct its Coulomb gauge, its field
  components and write the system \eqref{psieq} whose initial data is
  $\psi_0^\pm$. Invoking the uniqueness part of Theorem \ref{CT}, it
  follows that $\psi^\pm(t)$ are the gauge representation of $\mathcal
  W^\pm(t)$, hence the reconstruction in Proposition \ref{recprop}
  gives the Schr\"odinger Map $u(t)$ for each $t$.

\end{proof}

We can now identify the critical threshold for global well-posedness
and scattering. For any $m \geq 0$, we define $A(m)$ by
\[
A(m):= sup \{ S_{I_{max}}(\psi^-): M(\psi^-) \leq m \ \mbox{where
  $\psi^\pm$ is a solution to \eqref{psieq} satisfying}
\eqref{compnew} \}
\]
where $\psi^\pm$ is assumed to be a solution of \eqref{psieq},
satisfying the compatibility condition \eqref{compnew} and $I_{max}$
is its maximal interval of existence.

Obviously $A$ is a monotone increasing functions, it is bounded for
small $m$ by part iv) and it is left-continuous by part v) of Theorem
\ref{CT}. Therefore there exists a critical mass $0 < m_0 \leq +
\infty$ such that $A(m)$ is finite for all $m < m_0$ and it is
infinite $m \geq m_0$. Also any solution $\psi$ with $M(\psi) < m_0$
is globally defined and scatters.

Note that from \eqref{str} and the fact that $M(\psi^+)=M(\psi^-)$ (due to the compatibility relation),
it follows that we could have used
$S_{I_{max}}(\psi^+), M(\psi^+)$ in the definition of $A(m)$ and arrive to the
same conclusion as above with the same critical mass $m_0$.

\section{Concentration compactness} \label{concomp} 
The main goal of this section is to prove that if the above critical mass $m_0$
is finite, then there exists a critical element $\psi^\pm$ with mass $m_0$
which blows up, see Theorem \ref{cc}. Moreover we can be more precise about the
behavior of "scale" of the critical element, see Theorem \ref{lambdacontrol}.
The information provided by the two results aforementioned will be crucial
in the next section where we rule out the possibility that $m_0$ is finite.

We start by exhibiting the symmetries of the system \eqref{psieq}.  The system is
invariant under the time reversal transformation $\psi^\pm(r,t)
\rightarrow \psi^\pm(r,-t)$.  This allows us to focus our attention on
positive times, i.e. $t \geq 0$. 
Next, the system is invariant under two other transformations:
scaling, $\psi^\lambda = \lambda^{-1} \psi(\lambda^{-1} r,
\lambda^{-2} t)$ with $\lambda \in \R$, and phase multiplication,
$\psi^\alpha(r,t) = e^{i \alpha} \psi(r,t)$ with $\alpha \in \R/{2\pi
  \Z}$.  The phase multiplication can be ignored as the group
generated is compact. This way we generate the first (non-compact)
group $G$ of transformations $g_\lambda$ defined by
\[
g_{\lambda} f(r) = \lambda^{-1} f(\lambda^{-1} r)
\]
From \eqref{rela}, \eqref{A2} and \eqref{A0}, the effect of the action
$g_\lambda$ on $\psi^\pm$ is translated in the action of
$g_{\lambda}^1$ on $\psi_2,A_2$ and $g_\lambda^2$ on $A_0$ where
\[
g^1_{\lambda} f(r) = f(\lambda^{-1} r), \qquad g^2_{\lambda} f(r) =
\lambda^{-2} f(\lambda^{-1} r)
\]
The action of $g$ is extended to space-time functions by
\[
T_{g_\lambda} f(r,t) = \lambda^{-1} f(\lambda^{-1} r, \lambda^{-2} t)
\]

The equations in \eqref{psieq} are also time translation invariant and
this suggests enlarging the group $G$ to $G-$ as follows. Given
$\lambda > 0$ and $t \in \R$, we define
\[
g^-_{\lambda, t} f = \lambda^{-1} [e^{itH_{m - 1}} f](\lambda^{-1} r)
\]
We denote by $G'$ the group generated by these transformations. Given
two sequences $g^{n}, \tilde{g}^n \in G^-, \forall n \in \N$, we say
that they are asymptotically orthogonal iff
\begin{equation} \label{assort} \frac{\lambda_n}{\tilde \lambda_n} +
  \frac{\tilde \lambda_n}{\lambda_n} + | t_n \lambda_n^2 - \tilde t_n
  \tilde \lambda_n^2|= \infty
\end{equation}

We are now ready to state the two main results of this section.

\begin{theo} \label{cc} Assume that the critical mass $m_0$ is finite. Then
  there exists a critical element, i.e. a maximal-lifespan solution
  $\psi^\pm$ to \eqref{psieq} and satisfying \eqref{compnew}, with
  mass $m_0$ which blows up forward in time. In addition this solution
  has the following compactness property: there exists a continuous
  function $\lambda(t): I_+=[0,T_+) \rightarrow \R_+$ such that the
  sets
  \[
  K^\pm:= \left\{ \frac1{\lambda(t)} \psi^\pm(\frac{r}{\lambda(t)},t),
    t \in I_+ \right\}
  \]
  are precompact in $L^2$.
\end{theo}

\begin{rema} As a consequence of the compactness property it follows
  that there exists a function $C: \R^+ \rightarrow \R^+$ such that
  the above critical element satisfies
  \begin{equation} \label{psicon} \int_{r \geq C(\eta)
      \lambda(t)^{-1}} |\psi^\pm(t,r)|^2 rdr \leq \eta, \quad \forall
    t \in I_+.
  \end{equation}  

\end{rema}

One can construct critical elements whose function $\lambda(t)$ has
more explicit behavior.
\begin{theo} \label{lambdacontrol} Assume that the critical mass $m_0$ is finite. 
Then we can construct a critical element as in Theorem
  \ref{cc} such that one of the two scenarios holds true:

  i) $T_{+}=\infty$ and $\lambda(t) \geq c > 0, \forall t \geq 0$.

  ii) $T_+ < \infty$ and $\lim_{t \rightarrow T_+} \lambda(t)=\infty$.
\end{theo}

The proofs of the Theorems \ref{cc} and \ref{lambdacontrol} follow the same steps 
as their counterparts in \cite{BIKT2}, which in turn were inspired by the seminal work
of Kenig and Merle, see \cite{KeMe1}. We will not reproduce the proofs here due to their 
lengthy repetitive argument. Instead we state the intermediate Propositions which then lead 
to the proof of Theorem \ref{cc}. 

It is standard, see for instance \cite{KeMe1} and \cite{tvz} that the
result in Theorem \ref{cc} follows from the following
\begin{prop} \label{pcc} Assume $m_0 < +\infty$. Let $\psi_n^\pm:
  I_{n+}=[0,T_{n+}) \times \R \rightarrow \C, n \in \N$ be a sequence
  of solutions to \eqref{psieq}, satisfying \eqref{compnew} and such
  that $\lim_{n \rightarrow \infty} M(\psi^\pm_n)=m_0$ and $\lim_{n
    \rightarrow \infty} S_{I_{n+}}(\psi_n^\pm) = \infty$.  Then there
  are group elements $g_n \in G$ such that the sequence $g_n
  \psi_n^\pm(t_n)$ has a subsequence which converges in $L^2$.
\end{prop}

One of the main ingredients in the proof of Proposition \ref{pcc} is
the classical linear profile decomposition result.  These type of
results originate in the work of Bahouri and Gerard \cite{BG}, for the case of 
nonlinear wave equation and independently,  in the work of Merle and
Vega \cite{mv}, for the case of the nonlinear Schr\"odinger equation. For the case of
nonlinear Schr\"odinger equations see also \cite{bv}, \cite{Ker}, \cite{tvz}.

\begin{prop} \label{lprof} Let $\psi_0^n, n\in \N$ be a bounded sequence
  in $L^2$. Then (after passing to a subsequence if necessary) there
  exists a sequence $\phi^j, j \in \N$ of functions in $L^2$ and
  $g^{n,j} \in G^{-}, n,j \in \N$ such that we have the decomposition
  \begin{equation}
    \psi_0^n = \sum_{j=1}^l g^{n,j} \phi^j + w^{n,l}, \qquad \forall l \in \N
  \end{equation}
  where $w^{n,l}$ satisfies
  \begin{equation} \label{strvanish} \lim_{l \rightarrow \infty}
    \lim_{n \rightarrow \infty} S(e^{itH_{m-1}} w^{n,l}) = 0
  \end{equation}
  Moreover $g^{n,j}$ and $g^{n,j'}$ are asymptotically orthogonal for
  any $j \ne j'$ and we have the following orthogonality condition
  \begin{equation} \label{weak}
    \begin{split}
      weak \lim_{n \rightarrow \infty} (g^{n,j})^{-1} w^{n,l} =0,
      \quad \forall 1 \leq j \leq l
    \end{split}
  \end{equation}
  As a consequence the mass decoupling property holds
  \begin{equation} \label{massdec} \lim_{n \rightarrow \infty} (M(u^n)
    - \sum_{j=1}^l M(\phi^j)-M(w^{n,l}))=0
  \end{equation}
\end{prop}
A similar statement holds true also for the operator $H_{m +1}$.  We explained in
\cite{BIKT2} how this result follows as an equivariant counterpart of the result in Theorem 7.3
in \cite{tvz}. 

Based on Proposition \ref{lprof} and the results in the previous sections, one proves Proposition
\ref{pcc} by following the same steps as in \cite{BIKT2}. The details are left as an exercise.

\section{Momentum and localized momentum.} \label{MOM}

In this section we rule out the possible scenarios exhibited in
Theorem \ref{lambdacontrol}.  With the language used in Section
\ref{concomp}, we claim the following
\begin{theo} \label{lack} Critical elements do not exist.
\end{theo}

This will be based on virial type identities. Virial identities for
the Schr\"odinger Map problem originate in the work of Grillakis and
Stefanopoulos via a Lagrangian approach, see \cite{GM}. In their work
the formulation of these identities is at the level of the conformal
coordinate, obtained by using the stereographic projection. Our
approach is different in the sense that we derive the virial
identities at the level of the gauge components. However our results
can be derived from \cite{GM}.

\subsection{Virial type identities.}
This section is concerned with identities involving solutions of
\eqref{psieq} which satisfy the compatibility condition
\eqref{compnew}.

Given $a: \R_+ \rightarrow R$ a smooth function, i.e.
$|(r\partial_r)^\alpha a| \les_\alpha 1$, and which decays at infinity
we claim that
\begin{equation} \label{vir1} \frac{d}{dt} \int a(r) (A_2-m) r dr = 
  \int r \partial_r a(r) \Re( \psi_1 \frac{\bar{\psi}_2}r) rdr
\end{equation}

By using part i) of Theorem \ref{CT-CC}, the proof of \eqref{vir1}
goes as follows
\[
\begin{split}
  \frac{d}{dt} \int a(r) (A_2-m) r dr & = \int a(r) \partial_t A_2 rdr
  = - \int a(r) \Im{(\psi_0 \bar{\psi}_2)}  r dr \\
  & = - \int a(r)  \Im(i(\partial_r \psi_1 + \frac{1}r \psi_1 + \frac{iA_2}{r^2} \psi_2) \bar{\psi}_2) rdr \\
  & = - \int a(r)  \left( \Im(i \partial_r ( r \psi_1 \bar{\psi}_2)) - \Im(i r \psi_1 \partial_r \bar{\psi}_2) \right) dr \\
  & = \int \partial_r a(r) \Im(i \psi_1 \bar{\psi}_2) rdr = \int
  r \partial_r a(r) \Re( \psi_1 \frac{\bar{\psi}_2}r) rdr
\end{split}
\]
This computation is valid in a classical sense provided that $R_{m \pm 1}
\psi^\pm \in H^2$.  For general functions $\psi^\pm$ this is done by
using a regularization argument as we did in the proof of part i) of
Theorem \ref{CT}. Note that the quantities involved on both sides of
\eqref{vir1} are meaningful in light of the fact that $\psi_0 \in \dot
H^{-1}_e$ and $a \psi_2 \in \dHe$.

We now introduce the two momenta, the radial and the temporal one, as
follows
\[
M_1 = \frac{\Re{(\psi_1 \bar{\psi}_2)}}{A_2+m} , \qquad M_0= 
\frac{\Re{(\psi_0 \bar{\psi}_2)}}{A_2+m}
\]
Using the covariant calculus, the time momentum can be further written
as follows
\[
\begin{split}
  (A_2+m) M_0 &  =  \Re{(\psi_0 \bar{\psi}_2)} \\
  & = \Re{\left( i(D_1 \psi_1 + \frac{1}r \psi_1 + \frac{1}{r^2} D_2 \psi_2) \bar{\psi}_2 \right)} \\
  & = - \Im{(\partial_r  \psi_1 \bar{\psi}_2)} - \frac1r \Im{(\psi_1 \bar{\psi}_2)} - \frac{A_2}{r^2} |\psi_2|^2 \\
  & = - \partial_r \Im(\psi_1 \bar{\psi}_2) - \Im(\psi_1 \partial_r \bar{\psi}_2)  + \frac1r  \partial_r A_2 - \frac{A_2}{r^2} |\psi_2|^2 \\
  &=  \partial_r^2 A_2 + \frac1r \partial_r A_2 - A_2(|\psi_1|^2+
  \frac{|\psi_2|^2}{r^2})
\end{split}
\]
which leads to
\begin{equation}
  M_0= \Delta \ln(A_2+m) + \left(\frac{ \partial_r A_2}{A_2+m} \right)^2 - \frac{A_2}{A_2+m} (|\psi_1|^2+ \frac{|\psi_2|^2}{r^2}) 
\end{equation}

The following identity plays a fundamental role in our analysis
\begin{equation} \label{vir2}
  \partial_t M_1 - \partial_r M_0 = - \partial_r A_0
\end{equation}
This is established by using the covariant rules of calculus,
\[
\begin{split}
  \partial_t M_1 & = \frac{\Re{(D_0 \psi_1 \bar{ \psi}_2)}}{A_2+m} 
  + \frac{\Re{(\psi_1 \overline{ D_0 \psi_2})}}{A_2+m}
  - \frac{\Re{(\psi_1 \bar{\psi}_2)}}{(A_2+m)^2} \partial_t A_2 \\
  & = \frac{\Re{(D_1 \psi_0 \bar{ \psi}_2)}}{A_2+m} +
  \frac{\Re{(\psi_1 \overline{ D_2 \psi_0})}}{A_2+m}
  + \frac{\Re{(\psi_1 \bar{\psi}_2)}}{(A_2+m)^2} \Im{(\psi_0 \bar{\psi}_2)} \\
  & = \partial_r M_0 - \frac{\Re{(\psi_0 \partial_r \bar{
        \psi}_2)}}{A_2+m} + \frac{\Re{(\psi_0 \bar{
        \psi}_2)}}{(A_2+m)^2} \partial_r A_2 + \frac{\Re{(\psi_1
      \overline{ D_2 \psi_0})}}{A_2+m}
  + \frac{\Re{(\psi_1 \bar{\psi}_2)}}{(A_2+m)^2} \Im{(\psi_0 \bar{\psi}_2)} \\
  & = \partial_r M_0 - \frac{A_2 \Im{(\psi_0 \bar{ \psi}_1)}}{A_2+m} -
  \frac{\Re{(\psi_0 \bar{ \psi}_2)}}{(A_2+m)^2} \Im{(\psi_1
    \bar{\psi}_2)} + \frac{A_2 \Im{(\psi_1 \overline{\psi_0})}}{A_2+m}
  + \frac{\Re{(\psi_1 \bar{\psi}_2)}}{(A_2+m)^2} \Im{(\psi_0 \bar{\psi}_2)} \\
  & = \partial_r M_0 - 2 \frac{A_2 \Im{(\psi_0 \bar{ \psi}_1)}}{A_2+m}
  + \frac{|\psi_2|^2\Im{(\psi_0 \bar{\psi}_1)}}{(A_2+m)^2}  \\
  & = \partial_r M_0 - \Im{(\psi_0 \bar{ \psi}_1)} \\
  & = \partial_r M_0 -  \partial_r A_0 \\
\end{split}
\]
The above computation is meaningful provided that $R_{m \pm 1} \psi^\pm \in
H^3$.

Next we derive a localized version of \eqref{vir2} which has also the
advantage that it makes sense for $\psi^\pm \in L^2$ only. We take $a:
\R_+ \rightarrow \R$ to be a smooth function which decays at infinity
and satisfies also $|\frac1r \partial_r a | \les 1$ and
$|\partial_r^2 a| \les 1$. As a consequence we have that if $f \in \dHe$
then $\frac1r f \partial_r a \in \dHe$.

We multiply \eqref{vir2} by $a$ and integrate by parts as follows
\begin{equation} \label{vir3} \int a(r) M_1(r) dr \left|_0^T \right. +
  \int_0^T \int \partial_r a(r) M_0 dr = \int_0^T \int \partial_r a(r)
  A_0 dr
\end{equation}
This identity is now meaningful for $\psi^\pm \in L^2$. Indeed each
term is well-defined for the following reasons:

- the first since $a$ is bounded and $\frac1r M_1 \in L^2$,

- the second since $\psi_0 \in \dot H_e^{-1}$ and $\frac1r \partial_r
a \cdot \psi_2 \in \dHe$,

- the third since $\frac1r \partial_r a$ is bounded and $A_0 \in L^1$.

The justification of \eqref{vir3} for general $\psi^\pm \in L^2$ is
done by regularizing $\psi^\pm$ as above.

It will be useful to rewrite the second term on the left-hand side as
follows
\[
\begin{split}
  \int \partial_r a(r) M_0 dr & = \int \frac{1}r \partial_r a(r) \left( \Delta \ln(A_2+m) 
  + \left(\frac{ \partial_r A_2}{A_2+m} \right)^2 - \frac{A_2}{A_2+m} (|\psi_1|^2+ \frac{|\psi_2|^2}{r^2}) \right) rdr \\
  & =  - \int \partial_r (\frac{1}r \partial_r a(r)) \partial_r \ln(A_2+m) rdr 
- \int \frac{1}r \partial_r a(r)  G(r) dr
\end{split}
\]
where 
\[
G(r) =  - \left(\frac{ \partial_r A_2}{A_2+m} \right)^2 + \frac{A_2}{A_2+m} (|\psi_1|^2+
    \frac{|\psi_2|^2}{r^2}) 
\]
Using \eqref{comp1} one can easily see that $G$ is positive definite,
\begin{equation}\label{Gpos}
G
 \geq |\psi_1|^2 ( \frac{A_2}{A_2+m} - \frac{|\psi_2|^2}{(A_2+m)^2}) + \frac{A_2}{A_2+m}  \frac{|\psi_2|^2}{r^2} \geq \frac{m}{m+m_0} |\psi_1|^2 + \frac12  \frac{|\psi_2|^2}{r^2}
\end{equation}
where $m+m_0$ is an upper bound for $A_2$, obtained from  \eqref{A2}.

\subsection{Proof of Theorem \ref{lack}.} The argument is in the
spirit of the corresponding one in \cite{KeMe1}.

Based on a localized version of \eqref{vir1} and \eqref{vir3} we rule
out the possibilities exhibited in parts i) and ii) of Theorem
\ref{lambdacontrol}.

 By using \eqref{rela} and \eqref{loc2est}, the concentration
property \eqref{psicon} implies that all of the differentiated variables
$\psi_1, \psi_2$ and $A_2$ are concentrated in a compact set,
\begin{equation}\label{Rconc}
\int_{r \ges C(\eta) c^{-1} \eta^{-1}} (| \psi_1(r) |^2 +
\frac{|\psi_2(r)|^2}{r^2} + \frac{(A_2(r)-m)^2}{r^2}) rdr \les \eta,
\qquad \forall t \in I_+.
\end{equation}

We start by ruling out the existence of a critical element from part
i) of Theorem \ref{lambdacontrol}, i.e. the global element with
$\lambda(t) \geq c > 0, \forall t > 0$.  In \eqref{vir3}, we take
$a(r)=r^2 \phi(\frac{r}{R})$ where $\phi$ is smooth and equals $1$ for
$ r \leq 1$ and $0$ for $r \leq 2$, and obtain
\begin{equation} \label{virmid}
  \begin{split}
   \left. \int a(r) M_1(r) dr \right|_0^T  & = \int_0^T \!\! \int \partial_r  (\frac{1}r \partial_r a(r)) \partial_r \ln(A_2+m) rdr dt   - \int_0^T\!\! \int  \frac{1}r \partial_r a(r) G(r)  rdr  dt \\
    & + \int_0^T \int \partial_r a(r) A_0 dr dt
  \end{split}
\end{equation}
In this identity there are two main terms which we compare against
each other: the  one the left-hand side and the second on the
right-hand side. All the other terms are controlled by one of the two
main terms just mentioned.

We choose $\eta \ll 1$ small enough (the exact choice is derived from
the inequalities on the error terms below) and $R = C(\eta) c^{-1} \eta^{-1} \gg
c^{-1}$; we estimate the main terms in the above expression by
\[
\left|\int a(r) M_1 dr \right| \les \int r^2 |\psi_1| |\frac{\psi_2}r| rdr \les
R^2 \| \psi_1 \|_{L^2} \| \frac{\psi_2}r \|_{L^2} \les R^2 m_0
\]
which is valid both at $t=0$ and $t=T$, and, by \eqref{Gpos} and \eqref{Rconc}
\[
\int_0^T \int \frac{1}r \partial_r a(r) G(r)  rdr dt \ges T 
\]
By choosing $T \gg R^2 m_0$ we obtain a contradiction, provided that we
establish that all the other terms involved in \eqref{virmid} are of
error type.

The first term on the left-hand side of \eqref{virmid} is bounded as
follows
\[
\begin{split}
  \left|\int_0^T \int \partial_r (\frac{1}r \partial_r a(r)) \partial_r
  \ln(A_2+m) rdr dt\right| \les \int_0^T \int_{r \approx R} |\partial_r A_2|
  dr dt \les T \eta \ll T
\end{split}
\]
For the third term on the right-hand side of \eqref{virmid} we use
\eqref{A0c} and write
\[
\left|\int_0^T \int \partial_r a(r) A_0 dr dt\right| =\left|\int_0^T \int (-2+
\frac1r \partial_r a(r)) A_0 r dr dt\right|
\]
which is then bounded by
\[
\les \int_0^T \| \psi_1 \|_{L^2[R, \infty)} \| \frac{\psi_2}r
\|_{L^2[R, \infty)} dt \les T \eta \ll T
\]
We have just shown that the other two terms in \eqref{virmid} are of
error type and this finishes the contradiction argument. With this we
conclude ruling out the possibility exhibited in part i) of Theorem
\ref{lambdacontrol}.

Next we rule out the critical element of type exhibited in part
ii). In this case the assumption is that we have a critical element
with $T_+ <\infty, \lim_{t \rightarrow T_+} \lambda(t)=+\infty$.

For fixed $R$ we claim that
\begin{equation} \label{limitT} \lim_{t \rightarrow T_+} \int
  \phi(\frac{r}R) (A_2-m) rdr = 0
\end{equation}
Indeed, for given $\epsilon > 0$, pick $\eta$ such that $\eta^\frac12 R^2 <
\epsilon$.  Using \eqref{loc2est} we obtain
\[
\begin{split}
  & \| \phi(\frac{r}R) (A_2-m) \|_{L^1} \\
  \les & (C(\eta) \lambda(t)^{-1} \eta^{-1})^2 \| \frac{A_2-m}{r}
  \|_{L^2[0,C(\eta)\lambda^{-1}(t) \eta^{-1}]}
  + R^2 \| \frac{A_2-m}{r} \|_{L^2[C(\eta) \lambda^{-1}(t) \eta^{-1},R]}  \\
  \les & (C(\eta) \eta^{-1} \lambda(t)^{-1})^2 m_0^\frac12 + \eta^\frac12 R^2
\end{split}
\]
By choosing $t$ close enough to $T_+$, we obtain $ (C(\eta) \eta^{-1}
\lambda(t)^{-1})^2 m_0^\frac12 < \epsilon$, and this establishes \eqref{limitT}.

Next we choose $a(r)=\phi(\frac{r}R)$, fix $\eta > 0$, integrate
\eqref{vir1} on $[t,T_+)$ and use \eqref{limitT} to obtain
\[
\int \phi(\frac{x}R) (A_2(r,t)-m) rdr \les (T_+ - t) \| \psi_1(t)
\|_{L^2(|x| \approx R)} \| \frac{\psi_2(t)}r \|_{L^2(|x| \approx R)}
\les (T_+ - t) \eta
\]
provided that $R \ges C(\eta) \eta^{-1} \lambda(t)^{-1}$. By fixing $t$ and
taking $\eta \rightarrow 0$ (which also forces $R \rightarrow
\infty$), it follows that
\[
\int (A_2(r,t)-m) rdr = 0
\]
which implies $A_2(t) \equiv m$ hence, by \eqref{cons} and then by
\eqref{comp} it follows that $\psi_2(t) \equiv 0$ and $\psi_1(t)
\equiv 0$. Finally this implies by \eqref{rela} that $\psi^\pm(t) \equiv 0$
which contradicts the blow-up hypothesis at time $T_+$ (since the
solution is globally in time $\equiv 0$).

\section{Proof of the main result}

This section is dedicated to the proof of Theorem \ref{MT}. Given an
initial data $u_0 \in \dot H^1 \cap \dot H^3$, by using Theorem
\ref{clasic} it follows that it has a unique local solution on $[0,T]$
for some $T >0$.  On this interval we use sections \ref{CG} and
\ref{COG} to construct the associated compatible fields $\psi^\pm$ obeying the
system \eqref{psieq}. By using Theorem \ref{lack} (and the previous
reduction from Section \ref{concomp}) it follows that the solution
$\psi^\pm$ is globally defined on $[0,+\infty)$ and with $\| \psi^\pm
\|_{L^4(\R_+ \times \R_+)} < + \infty$. By part vii) of Theorem
\ref{CT} the $H^2$ regularity of $R_{m \pm 1} \psi^\pm_0$ is propagated at
all times $t \geq 0$. Invoking Proposition \ref{greg} this implies
that $u(t) \in \dot H^1 \cap \dot H^3$ with bounds depending on $\|
\psi^\pm \|_{L^4(\R_+ \times \R_+)} , \| R_{ m \pm 1} \psi^\pm_0 \|_{H^2}$
and $t$. Using again Theorem \ref{clasic}, this means that the
solution $u(t)$ can be continued past time $T$ and in fact for all
times $t \geq 0$ with $u(t) \in L^\infty_t (\R_+: \dot H^1 \cap \dot
H^3)$. The scattering statement refers to the scattering for
$\psi^\pm(t)$, which follows from the Cauchy theory for the system
\eqref{psieq}, see Theorem \ref{CT}.

Part ii) of the Theorem \ref{MT} is standard (see \cite{BIKT2} for details)
and it follows from \eqref{dpsi}, the Cauchy theory for the system \eqref{psieq}
and \eqref{udif}.

\section{Appendix}

\begin{proof}[Proof of Proposition \ref{greg}]
  We write the arguments below in a qualitative fashion in order to
  have a concise argument.  However one easily sees that the argument
  below provides quantitative bounds which lead to \eqref{rtr}.

  We first read the information $u \in \dot H^2$.  Using the equivariance property of $u$, we obtain
  \begin{equation} \label{uH2}
   H_m u_1, H_m u_2 \in L^2, H_0 u_3 \in L^2.
  \end{equation}
  Since $u_3^2=1+u_1^2+u_2^2$ it follows that
  \[
  \frac{u_1 \partial_r u_1 + u_2 \partial_r u_2}r =
  \frac{u_3 \partial_r u_3 }r \in L^2
  \]
  and by invoking $ \frac1r (\partial_r - \frac{m}{r}) (u_1,u_2) \in
  L^2$, we obtain $\frac{u_1^2+u_2^2}{r^2} \in L^2$.

  Since $D_r (v+iw)=0$ it follows that
  \[
  \partial_r \psi^\pm = \partial_r \left( \W^\pm \cdot (v + i w)
  \right) = (\partial_r \W^\pm) \cdot (v + i w)
  \]
  where we recall that
  \[
  \W^{\pm} = \partial_r u \pm \frac{1}{r} u \times \partial_\theta u \in T_u(\S^2)
  \]
  From this we compute
  \[
  \frac1r \W^\pm = \frac1r \left( (\partial_r  \mp \frac{m}r) u_1, (\partial_r
    \mp \frac{m}r) u_2, \partial_r u_3 \right) \pm m (\frac{(u_3-1)u_1}{r^2},
  \frac{(u_3-1)u_2}{r^2},-\frac{u_1^2+u_2^2}{r^2})
  \]
  From \eqref{uH2}, Lemma \ref{LBE} and  the fact that $\frac{u_1^2+u_2^2}{r^2} \in
  L^2$, it follows that $\frac{\W^\pm} r \in L^2$ if $m \geq 2$ and $\frac{\W^+}r \in L^2$ if $m=1$. 
  This implies the corresponding result for $\frac{\psi^\pm}r$. 
  
  A direct computation gives
  \[
  \begin{split}
    \partial_r \W^\pm & = \partial_r^2 u \pm m \partial_r (\frac{u \times Ru}r) \\
    & = \partial_r^2 u \mp m \frac{\partial_r u_3 \cdot u + u_3 \cdot \partial_r u}{r} \mp m \frac{\overrightarrow{k}- u_3 \cdot u}{r^2} \\
    & = (\partial_r^2 \mp \frac{m}r \partial_r \pm \frac{m}{r^2}) u \mp
    m \frac{u_3-1}r \partial_r u \mp m \frac{\overrightarrow{k}}{r^2} +
    f^\pm u
  \end{split}
  \]
  where
  \[
  f^\pm= \mp m \frac{\partial_r u_3}r \pm m \frac{u_3-1}{r^2}
  \]
  We then continue with
  \[
  \begin{split}
    \partial_r \psi^\pm & = \left( (\partial_r^2 \mp
      \frac{m}r \partial_r \pm \frac{m}{r^2}) u_1, (\partial_r^2 \mp
      \frac{m}r \partial_r \pm \frac{m}{r^2}) u_2,
      (\partial_r^2 \mp \frac{m}r \partial_r) u_3 \right) \cdot (v+iw) \\
    & \mp m \frac{u_3-1}r \psi_1   \pm i \frac{u_3-1}{r^2} \psi_2 \\
    & = F^\pm \mp m \frac{u_3-1}r  \frac{(m-1)\psi^+ + (m+1)\psi^-}{2}
  \end{split}
  \]
  where $F^\pm \in L^2$ from \eqref{uH2}. From the expression of $\W^\pm$
  and the Sobolev embeddings it follows that $\| \psi^\pm \|_{L^4} \les \| u \|_{\dot H^1 \cap \dot H^2}$,
  hence $\| \frac{\psi_2}r \|_{L^4} \les \| u \|_{\dot H^1 \cap \dot H^2}$. Therefore $\frac{u_3-1}{r}=
  \frac1{u_3+1} \frac{u_1^2 + u_2^2}{r}=  \frac1{u_3+1} \frac{|\psi_2|^2}{m^2 r} \in L^4$, which implies
  that $\frac{u_3-1}{r} \frac{(m-1)\psi^+ + (m+1)\psi^-}{2} \in L^2$ and we conclude with $\partial_r \psi^\pm \in L^2$.

  Hence we have just established that $R_{m \pm 1} \psi^\pm \in H^1$.  The
  procedure can be easily reversed, i.e. if $R_{m \pm 1} \psi^\pm \in H^1$
  then $u \in \dot H^2$, the details are left to the reader.

  Next we transfer third derivatives of $u$ to second derivatives for
  $\psi^\pm$ and vice-versa.  From $\Delta u \in \dot H^1$,
   using the equivariance properties of $u$, it follows
  \begin{equation} \label{uH3}
  H_{m} u_1, H_m u_2 \in \dHe, \partial_r H_0 u_3 \in L^2
  \end{equation}
  
   Using the above computation for
  $\partial_r \psi^+$, we have
  \[
  \begin{split}
    H_{m \pm 1} \psi^\pm & = (\partial_r+\frac1r) F^\pm \mp \frac{u_3-1}{r} \frac{(m-1) \partial_r \psi^+ + (m+1) \partial_r \psi^-}{2}
    \mp \frac{\partial_r A_2}{r} \frac{(m-1)\psi^+ + (m+1)\psi^-}{2m} \\
    & - \frac{(m \pm 1)^2}{r^2} \psi^\pm
  \end{split}
  \]
  The derivative in $\partial_r F^\pm$, can fall on either term in the expression of $F^\pm$. 
  From \eqref{uH3} and Lemma \ref{LBE} it follows that in all cases 
  $\partial_r (\partial_r^2 \mp \frac{m}r \partial_r \pm \frac{m}{r^2}) u_1,
  \partial_r (\partial_r^2 \mp \frac{m}r \partial_r \pm \frac{m}{r^2}) u_2,
  \partial_r (\partial_r^2 \mp \frac{m}r \partial_r) u_3  \in L^2$.  Using Lemma \ref{LBE},
  it follows that if $m \ne 1$, then $\partial_r u_1, \partial_r u_2, \partial_r u_3 \in \dHe \subset L^\infty$,
  hence by \eqref{cgeq} implies that $\partial_r v \in L^\infty$, and similarly $\partial_r w \in L^\infty$.
  If $m=1$ then by the same Lemma \ref{LBE}, $(\partial_r^2 \mp \frac{m}r \partial_r \pm \frac{m}{r^2}) u_1,
  (\partial_r^2 \mp \frac{m}r \partial_r \pm \frac{m}{r^2}) u_2,
  (\partial_r^2 \mp \frac{m}r \partial_r) u_3  \in \dHe \subset L^\infty$ and since $\partial_r u \in L^2$,
  then by \eqref{cgeq} $\partial_r v, \partial_r w \in L^2$. Hence we have completed the proof of the fact that 
  $\partial_r F^\pm \in L^2$. 
  
  Next, if $m = 1$, then from \eqref{uH3} and Lemma \ref{LBE} it follows
  that $\frac1r F^\pm \in L^2$. The other linear term left is $4 \frac{\psi^+}{r^2}$ (in the expression of $H_2 \psi^+$),
  which is estimated from
  \[
  \frac{\psi^+}{r^2} = \frac1{r^2} \left( (\partial_r  - \frac{m}r) u_1, (\partial_r
   - \frac{m}r) u_2, \partial_r u_3 \right) \cdot (v+iw) -\frac{u_1^2+u_2^2-u_3(u_3-1)}{r^3}(v_3+iw_3)
  \]
   Indeed, from Lemma \ref{LBE} it follows that $ (\partial_r  - \frac{m}r) u_1, (\partial_r
   - \frac{m}r) u_2, \partial_r u_3 \in L^2$, and from
   $|\frac{u_1^2+u_2^2-u_3(u_3-1)}{r^3}(v_3+iw_3)| \les \frac{|\psi_2|^3}{r^3}$
    and the Sobolev embedding $\frac{\psi_2}r \in L^6$ it follows that all the linear terms 
    in $H_{m\pm1} \psi^\pm \in L^2$. 
    
    If $m=2$, then $\frac{1}r F^+ \in L^2$ on behalf of Lemma \ref{LBE} and $\frac{\psi^+}{r^2} \in L^2$ 
    is shown as above. On the other hand, 
    \[
    \frac1r F^- - \frac1{r^2} \psi^- = \frac1r \left( (\partial_r^2 +
      \frac{1}r \partial_r - \frac{4}{r^2}) u_1, (\partial_r^2 +
      \frac{1}r \partial_r - \frac{4}{r^2}) u_2,
      (\partial_r^2 - \frac{1}r \partial_r) u_3 \right) \cdot (v+iw)
    \]
    belongs to $L^2$ on behalf of Lemma \ref{LBE}. 
    
    If $m \geq 3$, then it is a simple exercise to show that all the linear terms belong to $L^2$. 
    
    Moving on to the nonlinear terms in the expression of $H_{m \pm 1} \psi^\pm$, we notice
    that $\psi^\pm \in L^4 \cap L^6$ by using the Sobolev embeddings. Using \eqref{comp}, it then
    follows that  $\frac{\partial_r A_2}{r} \frac{(m-1)\psi^+ + (m+1)\psi^-}{2m}  \in L^2$ by using the $L^6$
    estimate for all terms involved. 
    
    For the last term we claim that $\partial_r \psi^\pm \in L^3$, from which
     $\frac{u_3-1}{r} \frac{(m-1) \partial_r \psi^+ + (m+1) \partial_r \psi^-}{2} \in L^2$ follows
     by using the $L^6$ estimate for $\frac{\psi_2}r$. The claim follows from the formula above 
     for $\partial_r \psi^\pm$, the $L^6$ estimate for $\psi^\pm$ and the Sobolev embedding
     $(\partial_r^2 \mp  \frac{m}r \partial_r \pm \frac{m}{r^2}) u_1, (\partial_r^2 \mp
      \frac{m}r \partial_r \pm \frac{m}{r^2}) u_2,
      (\partial_r^2 \mp \frac{m}r \partial_r) u_3 \in L^3$ (which can be derived using the Hankel calculus
      along the lines of the arguments in Lemma \ref{LBE}).

\end{proof}

\end{document}